\documentclass[11pt,reqno,oneside]{amsproc}
\title[Vorticity formulation for 3D Navier-Stokes equations]{On Green's function of the vorticity formulation for the 3D Navier-Stokes equations}

\author[I.~Kukavica]{Igor Kukavica}
\address{Department of Mathematics, University of Southern California, Los Angeles, CA 90089}
\email{kukavica@usc.edu}

\author[F.~Wang]{Fei Wang}
\address{School of Mathematical Sciences, CMA-Shanghai, Shanghai Jiao Tong University, 
  Shanghai 200240,China}
\email{fwang256@sjtu.edu.cn}

\author[Y.~Zhu]{Yichun Zhu}
\address{Academy of Mathematics and Systems Science, Chinese Academy of Sciences, Beijing 100190, China}
\email{steven00931002@hotmail.com}

\usepackage{fancyhdr}
\usepackage{comment}

  \chardef\forshowkeys=0
  \chardef\showllabel=0
  \chardef\refcheck=0
  \chardef\sketches=0
  \chardef\showcolors=0
  
\usepackage{marginnote}
\usepackage[margin=1in]{geometry}
\usepackage{amsmath, amsthm, amssymb}
\usepackage{times}
\usepackage{graphicx}
\usepackage[usenames,dvipsnames,svgnames,table]{xcolor}
\usepackage[colorlinks=true, pdfstartview=FitV, linkcolor=black, citecolor=black, urlcolor=black]{hyperref}

\ifnum\showcolors=1
  \def\colr{\color{red}}
  
  \def\colb{\color{black}}
  
  \definecolor{colorgggg}{rgb}{0.1,0.5,0.3}
  \definecolor{colorllll}{rgb}{0.0,0.7,0.0}
  \definecolor{colorhhhh}{rgb}{0.3,0.75,0.4}
  \definecolor{colorpppp}{rgb}{0.7,0.0,0.2}
  \definecolor{coloroooo}{rgb}{0.45,0.0,0.0}
  \definecolor{colorqqqq}{rgb}{0.1,0.7,0}

  \def\cole{\color{coloroooo}}
  
  \def\colu{\color{blue}}

  \definecolor{coloraaaa}{rgb}{0.6,0.6,0.6}
  
  \definecolor{mygray}{rgb}{.6,.6,.6}
\else
  \def\colr{\color{black}}
  
  \def\colb{\color{black}}

  \def\cole{\color{black}}
  
  \def\colu{\color{black}}

  \definecolor{mygray}{rgb}{0,0,0}
\fi

\ifnum\refcheck=1
  \usepackage{refcheck}
\fi

\ifnum\forshowkeys=1
  
  \usepackage[notref,notcite,color]{showkeys}
\fi

\begin{document}

\def\XX{X}
\def\YY{Y}
\def\ZZZ{Z}

\def\intint{\int\!\!\!\!\int}
\def\OO{\mathcal O}
\def\SS{\mathbb S}
\def\CC{\mathbb C}
\def\RR{\mathbb R}
\def\TT{\mathbb T}
\def\ZZ{\mathbb Z}
\def\z{\zeta}
\def\a{\alpha}
\def\g{\gamma}
\def\b{\beta}
\def\HH{\mathbb H}
\def\D{\mathcal D}
\def\d{\delta}
\def\RSZ{\mathcal R}
\def\LL{\mathcal L}
\def\SL{\LL^1}
\def\ZL{\LL^\infty}
\def\GG{\mathcal G}
\def\eps{\varepsilon}
\def\tt{\langle t\rangle}
\def\erf{\mathrm{Erf}}
\def\red#1{\textcolor{red}{#1}}
\def\blue#1{\textcolor{blue}{#1}}
\def\mgt#1{\textcolor{magenta}{#1}}
\def\ff{\rho}
\def\gg{G}
\def\tilde{\widetilde}
\def\sqrtnu{\sqrt{\nu}}
\def\ww{w}
\def\ft#1{#1_\xi}
\def\les{\lesssim}
\renewcommand*{\Re}{\ensuremath{\mathrm{{\mathbb R}e\,}}}
\renewcommand*{\Im}{\ensuremath{\mathrm{{\mathbb I}m\,}}}
\def\zz{\bar z}
\def\curl{\mathop{\rm curl}\nolimits}
\def\paren#1{\left( #1 \right)}
\def\diag{\text{diag}}

\newcommand{\norm}[1]{\left\|#1\right\|}
\newcommand{\nnorm}[1]{\lVert #1\rVert}
\newcommand{\abs}[1]{\left|#1\right|}
\newcommand{\NORM}[1]{|\!|\!| #1|\!|\!|}

\newtheorem{theorem}{Theorem}[section]
\newtheorem{corollary}[theorem]{Corollary}
\newtheorem{proposition}[theorem]{Proposition}
\newtheorem{lemma}[theorem]{Lemma}
\newtheorem{Hypotheses}[theorem]{Hypotheses}
\newtheorem{Hypothesis}{Hypothesis}

\theoremstyle{definition}
\newtheorem{definition}{Definition}[section]
\newtheorem{remark}[theorem]{Remark}

\ifnum\showllabel=1
 \def\llabel#1{\marginnote{\color{lightgray}\rm\small(#1)}[-0.0cm]\notag}
\else
\def\llabel#1{\notag}
\fi

\def\theequation{\thesection.\arabic{equation}}
\numberwithin{equation}{section}

\def\DeltaD{\Delta_{\text{D}}}
\def\DeltaN{\Delta_{\text{N}}}
\def\aand{\quad \text{and}\quad}
\def\re{\mathop{\rm \mathbb{R}e}\nolimits}
\def\im{\mathop{\rm \mathbb{I}m}\nolimits}
\def\Gammac{\Gamma_{\text{c}}}
\def\lec{\lesssim}
\def\DN{\mathop{\text{DN}}}
\def\ND{\mathop{\text{ND}}}
\def\inon#1{\hbox{\ \ \ \ \ \ \ }\hbox{#1}}                
\def\onon#1{\inon{on~$#1$}}
\def\inin#1{\inon{in~$#1$}}
\def\bnew{\colr}
\def\enew{\colb}
\def\bold{\colu}
\def\eold{\colb}
\def\ll{{\color{red}\ell}}
\def\ee{\epsilon_0}
\def\startnewsection#1#2{\section{#1}\label{#2}\setcounter{equation}{0}}   
\def\dist{\mathop{\rm dist}\nolimits}
\def\nnewpage{\newpage}
\def\sgn{\mathop{\rm sgn\,}\nolimits}    
\def\Tr{\mathop{\rm Tr}\nolimits}    
\def\div{\mathop{\rm div}\nolimits}  
\def\supp{\mathop{\rm supp}\nolimits}
\def\indeq{\quad{}}           
\def\period{.}                       
\def\semicolon{\,;}                  
\def\nts#1{{\colr #1 \colb}}
\def\ntsf#1{\footnote{\colr #1 \colb}}


\def\comma{ {\rm ,\qquad{}} }            
\def\commaone{ {\rm ,\quad{}} }          
\def\les{\lesssim}
\def\nts#1{{\color{red}\hbox{\bf ~#1~}}} 
\def\blackdot{{\color{red}{\hskip-.0truecm\rule[-1mm]{4mm}{4mm}\hskip.2truecm}}\hskip-.3truecm}
\def\bluedot{{\color{blue}{\hskip-.0truecm\rule[-1mm]{4mm}{4mm}\hskip.2truecm}}\hskip-.3truecm}
\def\purpledot{{\color{colorpppp}{\hskip-.0truecm\rule[-1mm]{4mm}{4mm}\hskip.2truecm}}\hskip-.3truecm}
\def\greendot{{\color{colorgggg}{\hskip-.0truecm\rule[-1mm]{4mm}{4mm}\hskip.2truecm}}\hskip-.3truecm}
\def\cyandot{{\color{cyan}{\hskip-.0truecm\rule[-1mm]{4mm}{4mm}\hskip.2truecm}}\hskip-.3truecm}
\def\reddot{{\color{red}{\hskip-.0truecm\rule[-1mm]{4mm}{4mm}\hskip.2truecm}}\hskip-.3truecm}
\def\tdot{{\color{green}{\hskip-.0truecm\rule[-.5mm]{3mm}{3mm}\hskip.2truecm}}\hskip-.1truecm}
\def\gdot{\greendot}
\def\bdot{\bluedot}
\def\pdot{\purpledot}
\def\ydot{\cyandot}
\def\rdot{\cyandot}

\def\fractext#1#2{{#1}/{#2}}
\def\ii{\hat\imath}
\def\fei#1{\textcolor{blue}{#1}}
\def\yichun#1{\textcolor{cyan}{#1}}
\def\igor#1{\textcolor{colorqqqq}{#1}}


\begin{abstract}
We give a novel vorticity formulation for the 3D Navier-Stokes equations with
Dirichlet boundary conditions. Via a resolvent argument, we obtain Green's
function and establish an upper bound, which is the 3D analog of~\cite{NN}.
Moreover, we prove similar results for the
corresponding Stokes problem with more general mixed boundary conditions.
\end{abstract}

\maketitle

\tableofcontents


\startnewsection{Introduction}{sec01}
We consider the Cauchy problem for the 3D incompressible Navier-Stokes equations 
  \begin{align} 
   &\partial_t u - \nu\Delta u + u\cdot\nabla u + \nabla p = 0 \label{EQ01}\\
   &\div u = 0 \label{EQ02}\\
   &u|_{t=0} = u_0 \label{EQ03}
  \end{align}
on the half-space
  \begin{equation}
   \mathbb H = \mathbb{T}^2\times \mathbb{R}_+
   ,
   \label{EQ04}
  \end{equation}
where $\mathbb{R}_+=[0,\infty)$, with  the {\em no-slip boundary condition} 
  \begin{align}
   &u|_{\partial \mathbb H} = 0
   \label{EQ05}.
  \end{align}
Here, $\nu>0$ is the kinematic viscosity. Formally setting $\nu = 0$ in \eqref{EQ01}--\eqref{EQ03} we arrive at the 3D incompressible Euler equations, with the so called {\em slip boundary condition} given by $u\cdot n|_{\partial \mathbb H}=0$. It is well-known that the vorticity plays an important role in studying the Navier-Stokes equations. For instance, in the domains without boundaries, the Navier-Stokes equations can be equivalently considered using the vorticity instead of the velocity. A benefit of such formulation is eliminating the pressure, which is replaced by the Bi\^ot-Savart law.  However, when considering the Navier-Stokes equations in the presence of boundaries, there is a difficulty with the boundary conditions for the vorticity, and thus the vorticity formulation is not as commonly used.  This difficulty was overcome by Anderson, who introduced in~\cite{A} the boundary conditions for the vorticity in the 2D~case.  Using Anderson's boundary conditions~\cite{A}, Maekawa in~\cite{M1} reformulated the 2D~Navier-Stokes equation using the vorticity by a semigroup approach. The reformulation allowed him to obtain a result stating that the inviscid limit holds in the half-space provided the initial vorticity is compactly supported.

The Anderson-Maekawa vorticity formulation \cite{A,M1} was further explored by Nguyen and Nguyen~\cite{NN}, who computed Green's function and obtained an upper bound for its derivatives.  This paper is the starting point for our main results. Namely, the main goal of the present paper is to obtain a new formulation of the 3D~Navier-Stokes vorticity setting (see \eqref{EQ41}--\eqref{EQ43}), derive the corresponding Green's function (see~\eqref{EQ233}), and obtain upper bounds for the Green's function (see~Theorem~\ref{T3}).  In particular, we show that the Navier-Stokes equations with no-slip boundary conditions can be written as
  \begin{align}
    &\omega_t + u\cdot\nabla\omega -\nu\Delta\omega =\omega\cdot\nabla u \label{EQ35a}\\
    &\nu(\partial_{z}+\Lambda_{\tau})\omega_\tau +\nu \nabla_\tau\Lambda_{\tau}^{-1} \nabla_\tau\cdot \omega_\tau
    \nonumber
    \\&\indeq
    =
     \partial_{z}(-\DeltaD)^{-1}(-u\cdot\nabla\omega_\tau+\omega\cdot\nabla u_\tau)
     + \nabla_\tau(-\DeltaN)^{-1}
     \paren{-u\cdot\nabla\omega_3+\omega\cdot\nabla u_z}
              \onon{\partial \mathbb{H}}
  ,
  \label{EQ36a}
  \\&
  \omega_3=0  \onon{\partial \mathbb{H}}
  ,
  \label{EQ37a}
  \end{align}
where $\tau$ indicates the tangential components;
see Section~\ref{sec02} for the definitions
of $(-\DeltaD)^{-1}$ and $(-\DeltaN)^{-1}$.
For the corresponding Stokes problem, we get a decomposition of the Green function mode by mode in Fourier frequency
  \begin{align}
	   G_{\xi}(t,y;z)= H_{\xi}(t,y;z) + R_{\xi}(t,y;z),
   \llabel{EQ113}
  \end{align}
where $H_{\xi}(t,y;z)$ is the heat kernel with homogeneous boundary condition in positive half line and the remainder term $R_{\xi}(t,y;z)$ satisfies a bound as the heat kernel with a fast decay tail (see \eqref{EQ128} and \eqref{EQ129} for more details).  After finding Green's function for the problem \eqref{EQ35a}--\eqref{EQ37a}, we obtain the upper bounds \eqref{EQ128} and \eqref{EQ129} where $G$ satisfies \eqref{EQ106}.  In 2D, such upper bounds were obtained by Nguyen and Nguyen in~\cite{NN}; the bounds proved to be extremely useful in the inviscid limit problem (\cite{BNNT,KVW3,KVW4,NN}).  The form of the upper bounds allows one to obtain the 3D analog of the result from \cite{KVW3,KVW4} on the inviscid limit for data that are analytic in a neighborhood of the boundary and Sobolev away from the boundary.

We point out that in the paper \cite{KM}, Kosaka and Maekawa obtained a vorticity boundary conditions, which agree with~\eqref{EQ36a} for the horizontal components of the vorticity; however, the condition $\partial_{3}\omega_3+\nabla'\cdot \omega'=0$ from \cite{KM} is replaced by the simple Dirichlet condition $\omega_3=0$.  Our simplification of the boundary conditions allows for a simpler treatment of the Green's function and easier derivation of the upper bounds stated in Theorem~\ref{T3}.  Besides simplifying the boundary conditions, the derivation of the conditions also appears to be new.  Instead of the explicit verification, we justify our conditions by a new uniqueness argument; see the proof of Theorem~\ref{T01} (see the argument after the proof of Lemma~\ref{L03}).  Finally, note that our last section provides an upper bound for the semigroup under general conditions on the Stokes problem.

The paper is structured as follows.  In Section~\ref{sec02}, we prove that all solutions of the Stokes system satisfy the Neumann/Dirichlet boundary conditions. Note that the basic step is the uniqueness argument presented at the end of the section.  Section~\ref{sec03} introduces the analytic semigroup using the resolvent inequality, stated in Theorem~\ref{T1}.  Section~\ref{sec04} introduces Green's function with Theorem~\ref{T2}, which is the sum of the classical heat Neumann and residual kernels.  The following statement, Theorem~\ref{T3} contains an upper bound for the spatial derivatives of the residual kernel. The proof of this statement is inspired by~\cite{NN}. The upper bounds apply to the inviscid limit problem for the half-space, which states that the inviscid limit holds locally in time for data that is analytic close to the boundary and Sobolev regular elsewhere. We do not present the details since they follow~\cite{KVW4} (see also \cite{W} for an alternative approach). In the last part of the section, we prove Duhamel's formula. Section~\ref{sec05} then contains the derivation of the resolvent and Green's function upper bounds in the case of general boundary conditions.

\startnewsection{Vorticity formulation}{sec02}

\subsection{A derivation of the vorticity formulation}
In this section, we provide the derivation of the new vorticity formulation.
Throughout the paper,
we use $\tau$ to denote the tangential component of a vector.
We denote by $(\DeltaD)^{-1}f$ the solution
$h$
to the Dirichlet problem
\begin{align*}
	-\Delta h &=f
   \inin{\mathbb{H}}
   \\
	\gamma(h)& =0,
\end{align*}
where both $f$ and $h$ are smooth and
with the necessary decay for $h$ at infinity;
we denote by $\gamma$ the trace of the function on $\partial\mathbb{H}$,
or $\mathbb{R}_+$, depending on the context.
Similarly, we denote by $(\DeltaN)^{-1}f$ the solution $h$ to the Neumann problem
\begin{align*}
	-\Delta h & = f \inin{\mathbb{H}} \\
	\gamma(\partial_{z} h) &= 0
\end{align*}
with
$\lim_{z \to \infty}h(z)=0$.
Also, we introduce the operator
  \begin{equation}
   \Phi(h)= \paren{(-\DeltaD)^{-1}h_\tau, (-\DeltaN)^{-1}h_n}
   .
   \label{EQ07}
     \end{equation}
Then we have the following lemma.
\cole
\begin{lemma}
\label{L01}
Assume that    $h\colon  \overline{\mathbb{H}} \to \mathbb{R}^3$,
 smooth
with compact support in $\overline{\mathbb{H}}$, is such that
$h=0$ on $\partial \mathbb{H}$ and $\div h=0$
in $\mathbb{H}$, then we have $\curl (\Phi(\curl h))=h$ in~$\mathbb{H}$.
\end{lemma}
\colb

\begin{proof}[Proof of Lemma~\ref{L01}]
First, observe that
  \begin{equation}\label{EQ08}
   \int_{\mathbb{H}} \curl h(x) \phi(x) dx
   = \int_{\mathbb{H}} h(x) \curl \phi(x)\,dx
   ,
  \end{equation}
for all $h,\phi \in \left(C_0^{\infty}(\overline{\mathbb{H}})\right)^3$ such that $h_{\tau}=0$ on~$\partial \mathbb{H}$.
Using the Helmholtz decomposition, for 
every $f \in  \left(C_0^{\infty}(\overline{\mathbb{H}})\right)^3$
there exist $\tilde{\phi}\colon  \overline{\mathbb{H}} \to \mathbb{R}$
and $\phi\colon \overline{\mathbb{H}} \to \mathbb{R}^3$ such that $f = \nabla \tilde{\phi} + \curl \phi$
and $(\curl \phi)_3=0$ on~$\partial \mathbb{H}$;
the functions $\phi$ and $\tilde\phi$ belong to all $W^{k,p}$,
where $k\in\mathbb{N}_0$ and~$p\in(1,\infty)$; see~\cite{McC}.
Thus we have
  \begin{equation}
   \int_{\mathbb{R}^3} \curl (\Phi(\curl h))(x) \cdot f(x) \,  dx = \int_{\mathbb{R}^3} \curl (\Phi(\curl h))(x)
   \cdot
   \bigl( \nabla \tilde{\phi} + \curl (\phi) \bigr)\,  dx 
   .
   \label{EQ09}
  \end{equation}
Let $g= \Phi (\curl h)$, and note that by the definition of $\Phi$, we have $g_{\tau}=0$ on~$\partial \mathbb{H}$. Since
  \begin{equation}
   \curl (\nabla\tilde{\phi})=0
   \llabel{EQ10}
   \end{equation}
 and 
  \begin{equation}
   \curl \curl \phi= -\Delta \phi + \nabla (\div \phi),
   \label{EQ11}
  \end{equation}
we obtain, using
\eqref{EQ08}, \eqref{EQ09}, and \eqref{EQ11},
  \begin{equation}
   \int_{\mathbb{R}^3} \curl (\Phi(\curl h))(x) \cdot f(x) \, dx = \int_{\mathbb{R}^3} \left(-\Delta g(x) + \nabla(\div g)(x)\right) \cdot \phi (x)  \, dx .
   \llabel{EQ12}
   \end{equation}
By the definition of $\Phi$, we have
  \begin{equation}
   -\Delta g(x)= \curl h.
   \llabel{EQ13}
    \end{equation}
Since $h_{\tau}\colb=0$ on  $\partial \mathbb{H}$,
we may use \eqref{EQ08} again to obtain
  \begin{equation}
  \int_{\mathbb{R}^3} \curl (\Phi(\curl h))(x) \cdot f(x) \, dx
  = \int_{\mathbb{R}^3} \bigl(h(x) \cdot \curl \phi(x)+ \nabla(\div g)(x) \cdot \phi (x)\bigr)  \,dx .
   \llabel{EQ14}
   \end{equation}
On the other hand, since $\Delta(\div g)= \div(\Delta g) = -\div (\curl h)=0$ and
\begin{equation}
g_{\tau}=0
   \aand {\partial_3} g_3=0
   ,
   \llabel{EQ15}
     \end{equation}
we get $\Delta (\div g)=0$ and
$\gamma(\div g)=0$,
which implies $\div g =0$. Hence, we have
\begin{equation}
\int_{\mathbb{R}^3} \curl (\Phi(\curl h))(x) \cdot f(x) \,dx
     = \int_{\mathbb{R}^3} h(x) \cdot \curl \phi(x)  \,dx
   \llabel{EQ16}
     \end{equation}
Finally, since $\div h=0$, we have
\begin{equation}
 \int_{\mathbb{R}^3} h(x) \cdot \nabla \tilde\phi(x) dx=0.
   \llabel{EQ17}
     \end{equation}
Therefore,
using also $f = \nabla \tilde{\phi} + \curl (\phi)$,
we conclude that
  \begin{equation}
   \int_{\mathbb{R}^3} \curl (\Phi(\curl h))(x) \cdot f(x) \, dx = \int_{\mathbb{R}^3}h(x) \cdot f(x) \, dx 
   ,
   \llabel{EQ18}
     \end{equation}
which proves the result.
\end{proof}

We also need the following identities.

\cole
\begin{lemma}
	\label{L02}
	For a smooth function $f$, we have
	\begin{equation}
		\gamma( \partial_{3} \DeltaD^{-1} \Delta f)
		= \gamma(\partial_{3} f)   +  \Lambda_{\tau} (\gamma(f))
		\label{EQ24}
	\end{equation}
	and
	\begin{equation}
		\gamma( \DeltaN^{-1} \Delta f)
		= \gamma(f) + \Lambda_{\tau}^{-1}(\gamma(\partial_3 f))
		,
		\label{EQ23}
	\end{equation}
	where $\Lambda_\tau=(-\Delta_\tau)^{1/2}$ and $\gamma$ is the trace operator restricted at $z=0$.
\end{lemma}
\colb

\begin{proof}[Proof of Lemma~\ref{L02}]
	The function
	$g=\Delta_{\text{D}}^{-1}\Delta f$ satisfies
	\begin{align}
		\begin{split}
			&\Delta g = \Delta f
			\\&
			\gamma(g)=0
			.
		\end{split}
		\llabel{EQ25}
	\end{align}
	Then
	\begin{align}
		\begin{split}
			&\Delta (g-f)=0
			\\&
			\gamma(g-f) = - \gamma(f)
			.
		\end{split}
		\llabel{EQ26}
	\end{align}
	By the definition of the Dirichlet-to-Neumann map $\DN$,
	we have
	$\gamma(\partial_{3}(g-f))
	= \DN(\gamma(f))$
	(note that the normal derivative on $\partial \mathbb{H}$ equals $-\partial_3$),
	from where we get
	\begin{align}
		\begin{split}
			\gamma(\partial_{3}g)
			&=
			\gamma(\partial_{3}f) + \DN(\gamma(f))
			=    \gamma(\partial_{3}f) + \Lambda_{\tau}(\gamma(f))
			,
		\end{split}
		\llabel{EQ27}
	\end{align}
	which gives \eqref{EQ24} since $\gamma(\partial_{3}g)$ equals
	the left-hand side in~\eqref{EQ24}.
	
	To get \eqref{EQ23}, we set
	$g=\DeltaN^{-1}\Delta f$ so that
	\begin{align}
		\begin{split}
			&\Delta g = \Delta f
			\\&
			\gamma(\partial_{3}g)=0
			.
		\end{split}
		\llabel{EQ28}
	\end{align}
	This leads to
	\begin{align}
		\begin{split}
			&\Delta (g-f)=0
			\\&
			\gamma(\partial_{3}(g-f)) = - \gamma(\partial_{3}f)
			.
		\end{split}
		\llabel{EQ29}
	\end{align}
	Denoting by $\ND=\DN^{-1}$ the Neumann-to-Dirichlet map, we have
	$\gamma((g-f))
	= \ND(\gamma(\partial_{3}f))$
	(again noting that the normal derivative and $\partial_3$ have opposite signs)
	and thus
	\begin{align}
		\begin{split}
			\gamma(g)
			&=
			\gamma(f) + \ND(\gamma(\partial_{3}f))
			=    \gamma(f) +  \Lambda_{\tau}^{-1}(\gamma(\partial_{3}f))
			,
		\end{split}
		\llabel{EQ30}
	\end{align}
	from where \eqref{EQ23} follows.
\end{proof}

\colb
Now,   we are ready to derive the vorticity formulation of the 3D Navier-Stokes equations.  
Applying the curl to the equation~\eqref{EQ01} gives
  \begin{equation}
	\omega_t + u\cdot\nabla\omega -\nu\Delta\omega
	 =\omega\cdot\nabla u
	 .
	\llabel{EQ06}
  \end{equation}
By
Lemma~\ref{L01}, the Bi\^ot-Savart law for the mixed Dirichlet-Neumann condition reads
\begin{equation}
	\label{EQ19}
	u=\curl W,
	\quad
	\text{where~}
	W= \Phi (\omega)
\end{equation}
We next obtain the boundary conditions for 
$\omega$.
For the third component of the vorticity, we have
  \begin{equation}
  \omega_3= \partial_{1}u_2-\partial_{2}u_1=0
  \onon{\partial \mathbb{H}}
   .
   \llabel{EQ20}
  \end{equation}
To get the boundary conditions for the horizontal component
$\omega_2$, first note that
  \begin{align}
    u_1 = \partial_{2} W_3 - \partial_{3} W_2
        = \partial_{2} (-\Delta_{\text{N}})^{-1} \omega_3
	    - \partial_{3} (-\Delta_{\text{D}})^{-1} \omega_2
   .
   \label{EQ21}
   \end{align}
Using \eqref{EQ01} and \eqref{EQ21}, we obtain
  \begin{align}
  \begin{split}
  -\partial_{t}u_1
  &=
   \bigl(
    \partial_{3}(-\Delta_\text{D})^{-1}\partial_{t}\omega_2
    \bigr)
    -
   \bigl(
    \partial_{2} (-\Delta_{\text{N}})^{-1} \partial_{t}\omega_3
   \bigr)
  \\&
  =
   \Bigl(
     \partial_{3}(-\Delta_D)^{-1}(-u\cdot\nabla\omega_2+\omega\cdot\nabla u_2+\nu\Delta\omega_2)
   \Bigr)
  \\&\indeq
    -
   \Bigl(
    \partial_{2}(-\Delta_N)^{-1}(-u\cdot\nabla\omega_3+\omega\cdot\nabla u_3+\nu\Delta\omega_3)
   \Bigr)
   \onon{\partial \mathbb{H}}
   .
  \end{split}
  \label{EQ22}
  \end{align}
By \eqref{EQ24}, we have
  \begin{equation}
   \gamma(   \partial_{3}(-\Delta_D)^{-1}\Delta\omega_2)
   =
   -\gamma(\partial_{3}\omega_2+\Lambda_\tau\omega_2)
  \label{EQ31}
  \end{equation}
while \eqref{EQ23} implies
\begin{equation}
\gamma(   \partial_{2}(-\Delta_N)^{-1}\Delta\omega_3)
   =
   -\gamma(\partial_{2}\omega_3 
      +\partial_2\Lambda_{\tau}^{-1}\partial_3 \omega_3).
  \label{EQ32}
  \end{equation}
Using $\gamma(u_1)=0$, along with \eqref{EQ31} and \eqref{EQ32} in \eqref{EQ22},
we get
  \begin{align}
  \begin{split}
  &
  \nu\gamma((\partial_{3}+\Lambda_\tau)\omega_2)
    - \nu \gamma(\partial_2\Lambda_{\tau}^{-1} \partial_3 \omega_3)
  \\&\indeq
  =\gamma\bigl(\partial_{3}(-\Delta_D)^{-1}(-u\cdot\nabla\omega_2+\omega\cdot\nabla u_2)\bigr)
   -\gamma\bigl(\partial_{2}(-\Delta_N)^{-1}(-u\cdot\nabla\omega_3+\omega\cdot\nabla u_3)\bigr)
  .
  \end{split}
  \label{EQ33}
  \end{align}  
Analogously,
we obtain the boundary condition for $\omega_1$ as
  \begin{align}
  \begin{split}
  &
  \nu\gamma((\partial_{3}+\Lambda_\tau)\omega_1)
    - \nu \gamma(\partial_1\Lambda_{\tau}^{-1} \partial_3 \omega_3)
  \\&\indeq
  =\gamma\bigl(\partial_{3}(-\Delta_D)^{-1}(-u\cdot\nabla\omega_1+\omega\cdot\nabla u_1)\bigr)
   -\gamma\bigl(\partial_{2}(-\Delta_N)^{-1}(-u\cdot\nabla\omega_3+\omega\cdot\nabla u_3)\bigr)
  .
  \end{split}
   \label{EQ34}
  \end{align}  
Using also the divergence-free condition 
$\partial_{3}\omega_3=-\nabla_{\tau}\cdot \omega_{\tau}$,
in \eqref{EQ33} and~\eqref{EQ34}, we obtain
the vorticity formulation for the 3D~Navier-Stokes equations with mixed boundary conditions,
  \begin{align}
    &\omega_t + u\cdot\nabla\omega -\nu\Delta\omega =\omega\cdot\nabla u \label{EQ35}\\
    &\nu(\partial_{z}+\Lambda_{\tau})\omega_\tau +\nu \nabla_\tau\Lambda_{\tau}^{-1} \nabla_\tau\cdot \omega_\tau
    \nonumber
    \\&\indeq
    =
     \partial_{z}(-\Delta_D)^{-1}(-u\cdot\nabla\omega_\tau+\omega\cdot\nabla u_\tau)
     + \nabla_\tau(-\Delta_N)^{-1}
     \paren{-u\cdot\nabla\omega_3+\omega\cdot\nabla u_z}
              \onon{\partial \mathbb{H}}
  ,
  \label{EQ36}
  \\&
  \omega_3=0  \onon{\partial \mathbb{H}}.     \label{EQ37}
  \end{align}
\colb

\subsection{Equivalence of the vorticity and velocity formulations}
Next, we show that the vorticity system obtained above is equivalent to the original Navier-Stokes system. 

\cole
\begin{theorem}
\label{main:th:2}
Consider the Navier-Stokes system \eqref{EQ01}--\eqref{EQ03}
with the initial data $u_0\in H^{2}(\mathbb{H})$
with the classical compatibility conditions.
Then the system is equivalent to \eqref{EQ35}--\eqref{EQ37}.
\end{theorem}
\colb

Note that the $H^{2}(\mathbb H)$ initial velocity does
not require any compatibility condition
other than $\div u_0=0$ and
$\gamma(u_0\cdot n)=0$ (see~\cite{T}).
This guarantees that $u\in L^{2}([0,T],H^{3}(\HH))$, which allows us to
define the trace of~$\partial_{z}\omega$.
Note that the $H^{2}$ regularity of the vorticity is sufficient to define the trace of
$\gamma(\partial_{z} \omega_{\tau,\xi})$
for~$t>0$.

We prove Theorem~\ref{main:th:2} by considering, more generally,
the Stokes problem
\begin{subequations}
	\label{sto:vel}
\begin{align} 
	&\partial_t u - \nu\Delta u + \nabla p = f \label{EQ01a}\\
	& \div u = 0 \label{EQ02a}\\
	&u=0 \onon{\partial \mathbb{R}_{+}}\\
	&u|_{t=0} = u_0 \label{EQ03a}
\end{align}
\end{subequations}
and its vorticity formulation
\begin{subequations}
	\label{sto:vor}
\begin{align}
	&\omega_t  -\nu\Delta\omega = \curl f \label{EQ225}\\
	&\nu(\partial_{z}+\Lambda_{\tau})\omega_\tau +\nu \nabla_\tau\Lambda_{\tau}^{-1} \nabla_\tau\cdot \omega_\tau
	= K_{\tau} (\curl f)
	\onon{\partial \mathbb{R}_{+}}
	,
	\label{EQ226}
	\\&
	\omega_3=0  \onon{\partial \mathbb{R}_{+}}.
	\label{EQ229}
	\\&
	\omega|_{t=0} = \curl u_0 
	\label{EQ227}
\end{align}
\end{subequations}
where
\begin{equation}
	K_{\tau}(g)
	= 
	\partial_{z}(-\Delta_D)^{-1}g_{\tau}
	+ \nabla_\tau(-\Delta_N)^{-1}
	g_3
	,
	\llabel{EQ228}
\end{equation}
assuming a mild decay of $u_0$, $f$, $u$, and $\omega$ in~$z$.
Taking the Fourier transform of~\eqref{sto:vor} with respect to the tangential variables gives
  \begin{subequations}
  \label{EQ40}
  \begin{align}
    &
    \partial_{t}\omega_\xi - \nu\Delta_\xi\omega_\xi = (\curl f)_\xi \label{EQ41}
    \\&
    -\nu(\partial_{3}+|\xi|)\omega_{\tau,\xi} + \nu \xi|\xi|^{-1}\xi\cdot \omega_{\tau, \xi}= (K_{\tau} (\curl f))_{\xi}
    \onon{\partial \mathbb{R}_+}
    \label{EQ42}
    \\&
    \omega_{3,\xi} = 0
    \onon{\partial \mathbb{R}_+}
    ,
    \label{EQ43}
  \end{align}
  \end{subequations}
for which we may derive uniqueness of solutions, where
  \begin{equation}
   \Delta_\xi= -|\xi|^2+\partial_{z}^2   
   .
   \label{EQ239}
  \end{equation}
In the most of the paper, we consider $\xi\in \mathbb{Z}^{2}$ fixed.
For the Stokes problem, we assume that
$(\curl f)_\xi$ and $(K_{\tau} (\curl f))_{\xi}$ are sufficiently smooth with a sufficient decay at
infinity. For instance, we can take
$(\curl f)_\xi \in L^{2}([0,T],L^2(\mathbb{R}_{+}))$
and
$\gamma(K_{\tau} (\curl f)_{\xi}|)$
continuous in time.

\cole
\begin{lemma}
\label{L03}
There exists a unique solution in $L^\infty([0, T], H^{1}(\mathbb{R}_{+}))\cap L^{2}([0,T],H^{2}(\mathbb{R}_{+}))$ to the Stokes problem~\eqref{EQ40} with
the initial vorticity in $H^{1}(\mathbb{R}_{+})$.
\end{lemma}
\colb

\begin{proof}[Proof of Lemma~\ref{L03}]
The existence is due to the classical existence result of Stokes system and we then turn to the uniqueness part. 
It suffices to prove that $0$ is the only solution
of the equation
\begin{subequations}\label{EQ44}
  \begin{align}
    \partial_{t}\omega_\xi - \nu\Delta_\xi\omega_\xi &= 0\label{EQ45}\\
    -\nu(\partial_{z}+|\xi|)\omega_{\tau,\xi}|_{z=0} + \nu \xi|\xi|^{-1}\xi\cdot \omega_{\tau, \xi}|_{z=0}&=0
    \label{EQ46}
     \\
    \omega_{3,\xi}|_{z=0} & = 0,\label{EQ47}
  \end{align}
\end{subequations}
with the initial condition $0$ for every $\xi\in\mathbb{Z}^2$
and with the velocity in $L^{\infty}([0,T],H^{2}(\mathbb{R}_{+}))\cap L^{2}([0,T],H^{3}(\mathbb{R}_{+}))$.
From 
  \begin{align}
  \begin{split}
   &
   \partial_{t}\omega_{3,\xi} - \nu\Delta_\xi\omega_{3,\xi} =0
   \\&
   \omega_{3,\xi}|_{z=0} = 0
   ,
  \end{split}
   \llabel{EQ48}
  \end{align}
we first obtain that $\omega_{3,\xi}=0$.

Taking a dot product with $\xi$ with the equations for $\omega_{1,\xi}$
and $\omega_{2,\xi}$ in \eqref{EQ45}, we obtain
  \begin{align}
   \partial_{t} (\xi \cdot \omega_{\tau,\xi}) - \nu\Delta_\xi (\xi \cdot\omega_{\tau,\xi}) = 0
    .
   \label{EQ49}
  \end{align}
Then, multiplying both sides of \eqref{EQ46} with $\xi$
leads to
  \begin{align}
   -\nu(\partial_{z}+|\xi|)\xi \cdot \omega_{\tau,\xi}
   + \nu|\xi|\xi\cdot \omega_{\tau, \xi}
   =0
   \onon{\partial \mathbb{R}_{+}}
   ,
  \llabel{EQ50}
  \end{align}
which gives, after a cancellation,
  \begin{align}
   \partial_{z} (\xi \cdot \omega_{\tau,\xi}) =0
   \onon{\partial \mathbb{R}_{+}}
   .
   \label{EQ51}
  \end{align}
Combining
\eqref{EQ49}, \eqref{EQ51}, and
$\xi \cdot \omega_{\tau,\xi}|_{t=0}=0$
gives
  \begin{equation}
   \xi \cdot \omega_{\tau,\xi}=0
   \label{EQ52}
  \end{equation}
for $t\geq0$.
Using \eqref{EQ52} in \eqref{EQ46}, we obtain
  \begin{align}
   \begin{split}
   &\partial_{t} \omega_{\tau,\xi} - \nu \Delta_{\xi} \omega_{\tau,\xi}=0
   \\&
   (\partial_z + |\xi|)\omega_{\tau,\xi}=0
   \onon{\partial \mathbb{R}_{+}}
   .
  \end{split}
   \llabel{EQ53}
  \end{align}
Taking an inner product
of \eqref{EQ49} with $\omega_{\tau,\xi}$, we get
  \begin{align}
    \begin{split}
   &
   \frac{1}{2\nu}\frac{d}{dt}\left\|\omega_{\tau,\xi}\right\|^2_{L^2(\mathbb{R}_+)}
    =
    \int_{\mathbb{R}_+} (-|\xi|^2 + \partial_z^2) \omega_{\tau,\xi}(z) \cdot \omega_{\tau,\xi}(z) dz
   \\&\indeq    
    =
  -|\xi|^2\cdot \left\|\omega_{\tau,\xi}\right\|_{L^2(\mathbb{R}_+)}^2
  - |\xi|\left| \omega_{\tau,\xi}(0) \right|^2
  - \int_{\mathbb{R}^+} \left| \partial_3 \omega_{\tau,\xi}(z) \right|^2 dz \leq 0
  ,
  \end{split}
   \llabel{EQ54}
  \end{align}
where we integrated by parts in the second step.
This implies that $\omega_{\tau,\xi}=0$, concluding the proof.
\end{proof}

\cole
\begin{theorem}
\label{T01}
Consider the Stokes problem \eqref{sto:vel}
with the initial data $u_0\in H^{2}(\mathbb{H})$
and the force $f\in L^{2}([0,T],H^{1}(\HH))$,
with the classical compatibility conditions.
Then the system is equivalent to~\eqref{sto:vor}.
\end{theorem}
\colb

\begin{proof}[Proof of Theorem~\ref{T01}]
The argument above, starting with \eqref{EQ19}, showing if $u$ solves \eqref{EQ01}--\eqref{EQ05},
then $\omega=\curl u$ solves \eqref{EQ35}--\eqref{EQ37},
carries over to the Stokes case. Namely, the argument shows that if $u$ is a solution of \eqref{EQ01a}--\eqref{EQ03a},
then $\omega=\curl u$ solves \eqref{EQ225}--\eqref{EQ227}.
Conversely, assuming that $\omega$ solves the equations \eqref{EQ225}--\eqref{EQ227},
and letting $u=\Phi(\omega)$,
we claim that $u$ solves \eqref{EQ01a}--\eqref{EQ03a}.
Consider the solution $\tilde u$ of the problem \eqref{EQ01a}--\eqref{EQ03a}; we claim that
$u=\tilde u$. The vorticity $\tilde \omega$ corresponding to
$\tilde u$, i.e.,
$\tilde \omega=\curl \tilde u$, satisfies
\eqref{EQ225}--\eqref{EQ226}. By Lemma~\ref{L03}, we obtain
$\omega=\tilde \omega$, from there
$u=\Phi(\curl \omega)=\Phi(\curl \tilde \omega)=\tilde u$,
which gives $u=\tilde u$, as claimed.
\end{proof}
\begin{proof}[Proof of Theorem~\ref{main:th:2}]
The proof follows directly from Theorem~\ref{T01} by setting $f= -u\cdot\nabla u$ in the equations~\eqref{sto:vel}.
\end{proof}

\startnewsection{Analytic semigroup}{sec03}
In this section, we study the horizontal components of the system~\eqref{sto:vor}, which leads to the resolvent equations
  \begin{align}
  \begin{split}
   &
  \lambda \omega_\xi  -\nu \Delta_{\xi} \omega_\xi = f
  \\&
    -(\partial_{3}+|\xi|)\omega_{\xi}|_{z=0} +  \xi|\xi|^{-1}\xi\cdot \omega_{\xi}|_{z=0}= 0
   ,
  \end{split}
   \llabel{EQ55}
  \end{align}
for fixed $\xi$. And the vertical component satisfies the heat equation with the Dirichlet boundary condition, for which the analysis is standard.

\begin{definition}
Let $A_{\xi}$ be the realization of the Laplace operator $\nu \Delta_{\xi}$ with the boundary condition
  \begin{equation}
   -\gamma(\partial_{3}+|\xi|)u +  \gamma\xi|\xi|^{-1}\xi\cdot u= 0,
  \llabel{EQ56}
  \end{equation}
in $L^2(\mathbb{R}_+)$, where $\gamma$ denotes the trace operator at $\{z=0\}$. We define the domain of $A_{\xi}$ by
  \begin{equation}
   D(A_{\xi}):= \left\{u \in H^2(\mathbb{R}_+): -\gamma(\partial_{3}+|\xi|)u +  \gamma\xi|\xi|^{-1}\xi\cdot u= 0\right\}.
   \llabel{EQ57}
     \end{equation}
\end{definition}
Above and in the sequel, we do not distinguish in notation between vector and scalar valued function spaces; for example, we write
$L^{2}(\mathbb{R}_+)$ for
$(L^{2}(\mathbb{R}_+))^2$.

\subsection{Resolvent equation in $L^2$}\label{S3.1}
We first solve the resolvent equation 
  \begin{align}\label{EQ59}
   (\lambda - \nu \Delta_{\xi})u &=f, \\
   -\gamma(\partial_{3}+|\xi|)u +  \gamma\xi|\xi|^{-1}\xi\cdot u&= 0. \label{EQ60}
  \end{align}
for $f\in L^2{(\mathbb{R}_+)}$, where $\xi\in\mathbb{Z}$ is fixed and
$\lambda \in (-|\xi|^2 \nu, +\infty)\setminus \{0\}$.
It is straight-forward to check that the results obtained below extend
to~$\lambda \in \mathbb{C} \setminus ((-\infty,-|\xi|^2\nu]\cup \{0\})$.
To start, we extend $f$ evenly to $\bar{f}$ on $\mathbb{R}$ by
  \begin{align}
  \bar{f}(z)= \begin{cases}
    f(z), & z\in [0,+\infty)\\
    f(-z), &z\in (-\infty, 0]
  \end{cases}
   \llabel{EQ61}
  \end{align}
and solve the non-homogeneous equation 
  \begin{align}
   (\lambda - \nu \Delta_{\xi}) v =\bar{f}
   \label{EQ62}
  \end{align}
in~$L^2(\mathbb{R})$.
Since $f\in L^2(\mathbb{R}_+)$ and thus
$\bar f\in L^{2}(\mathbb{R})$, we may take the Fourier transform of \eqref{EQ62} in $z$ and obtain 
  \begin{align}
   (\lambda+\nu( |\xi|^2+ |\zeta|^2))\mathcal{F}(v)=\mathcal{F}(\bar{f}),
   \label{EQ63}
  \end{align}
with $\zeta$ denoting the corresponding Fourier variable. Hence, we get
  \begin{equation}
   {v}= \mathcal{F}^{-1}\left((\lambda+ \nu(|\xi|^2 + |\zeta|^2))^{-1} \mathcal{F}(\bar{f}) \right)= \mathcal{F}^{-1}\left((\lambda+ \nu( |\xi|^2+ |\zeta|^2))^{-1}\right) \star \bar{f}
   ,
   \llabel{EQ64}
   \end{equation}
where $\star$ denotes the convolution in $z$ variable.
Letting
  \begin{equation}
   \mu=\nu^{-1/2} \sqrt{\lambda + \nu|\xi|^2 }
   \label{EQ80}
  \end{equation}
results in
  \begin{equation}
  \begin{aligned}
   \mathcal{F}^{-1}\left((\lambda+ \nu(|\xi|^2+ |\zeta|^2))^{-1}\right)(z)
    &= \frac{1}{\nu } \int_{\mathbb{R}} \frac{1}{\mu^2+ |\zeta|^2}e^{ i z \cdot \zeta} \,d\zeta
    = \frac{1}{\nu \mu } \int_{-\infty}^{\infty} \frac{1}{1+ \zeta^2} e^{ i z \mu\cdot \zeta} \,d\zeta
.
   \end{aligned}
   \llabel{EQ65}
  \end{equation}
Recalling the Fourier transform of the Poisson kernel, we have
  \begin{equation}
  \mathcal{F}^{-1}\left((\lambda+ \nu(|\xi|^2+ |\zeta|^2))^{-1}\right)(z) =\frac{1}{2\nu \mu}  e^{- \mu |z|}   .
   \llabel{EQ66}
   \end{equation}
Therefore, we obtain 
  \begin{align}
  \begin{split}
   v(z)
   &= \frac{1}{2\nu \mu}\int_{-\infty}^{\infty}   e^{- \mu |z'-z|} \bar{f}(z') dz'
   =  \frac{1}{2\nu \mu}\int_{0}^{\infty}   \left(e^{- \mu |z'-z|} + e^{- \mu |z+z'|}\right) {f}(z') dz'
  .
  \end{split}
  \llabel{EQ67}
  \end{align}
The difference $w=u-v$ satisfies
  \begin{align}
    \begin{split}
   &
   (\lambda - \nu \Delta_{\xi})w =0
   \\&
   -\gamma(\partial_{z}+|\xi|)w+  \gamma\xi|\xi|^{-1}\xi\cdot w 
   = \gamma(\partial_{z}+|\xi|)v -  \gamma(\xi|\xi|^{-1}\xi\cdot v)
  \onon{\partial \mathbb{R}_{+}}
   .
  \end{split}
  \label{EQ69}
  \end{align}
In order to obtain a solution in $L^2(\mathbb{R}_{+})$,
we note that, based on the first equation in \eqref{EQ69}, $w$ has the form
  \begin{align}
   w(z)=  c_0 e^{-\mu z}
   ,
   \label{EQ70}
  \end{align}
where $c_0 \in \mathbb{C}^2$.
To find $c_0$, we first compute the value of $(\partial_{z}+|\xi|) w +  \xi|\xi|^{-1}\xi\cdot w$ at the boundary,
  \begin{align}
    \begin{split}
   \gamma(
   -(\partial_{z}+|\xi|) w +  \xi|\xi|^{-1}\xi\cdot w)
   &=
   \gamma(\mu c_0 e^{-\mu z} - |\xi| c_0 e^{\mu z} + \xi|\xi|^{-1} \xi \cdot c_0 e^{\mu z})\\
   &= \mu c_0 -|\xi| c_0 +  \xi|\xi|^{-1} \xi \cdot c_0 .
  \end{split}
   \label{EQ71}
  \end{align}
From \eqref{EQ71}, we thus obtain
  \begin{align}\label{EQ72}
   Bc_0 =
   \gamma(   (\partial_{z}+|\xi|)v -  \xi|\xi|^{-1}\xi\cdot v)
   ,
  \end{align}
where the matrix $B$ is given by
  \begin{align}B=
  \begin{pmatrix}
  \mu-|\xi| + \xi_1^2 |\xi|^{-1} & \xi_1\xi_2 |\xi|^{-1}\\
  \xi_1\xi_2 |\xi|^{-1} & \mu -|\xi| + \xi_2^2 |\xi|^{-1}\\
  \end{pmatrix}
  .
  \label{EQ73}
  \end{align}
It is easy to check 
\begin{align*}
 \det B = \mu(\mu-|\xi|),
\end{align*}
and hence $B$ is invertible when
$\mu\neq 0$ and $\mu\neq |\xi|$, which holds
if $\lambda\neq -\nu\xi^2$ and $\lambda\neq 0$. 
We then obtain
$c_0$ from \eqref{EQ72}, which then gives $w$ by~\eqref{EQ70}.
Then 
$u=v+w$ satisfies the system~\eqref{EQ59} with the boundary condition~\eqref{EQ60}.
For a solution $u\in L^2(\mathbb{R}_+)$, denoting
\begin{align*}
	u=R(\lambda;A_{\xi})f:=(\lambda-A_{\xi})^{-1}f,
\end{align*}
then we have
  \begin{align}
  \begin{split}
   &
   R(\lambda;A_{\xi})f
   =
   v
      +
   w
  \end{split}
   \llabel{EQ234}
  \end{align}
where
  \begin{equation}
   v(y)
   =
   \frac{1}{2\nu \mu}\int_{0}^{\infty}   \left(e^{- \mu |z-y|} + e^{- \mu |z+y|}\right) {f}(z) dz
   \label{EQ58}
  \end{equation}
and
  \begin{equation}
   w(y)=B^{-1}  \gamma((\partial_{3}+|\xi|)v -  \xi|\xi|^{-1}\xi\cdot v)  e^{-\mu y}
   ,
   \label{EQ68}
   \end{equation}
with $\mu$ as in~\eqref{EQ80}.

Next, we show that the operator $A_\xi$ is sectorial and thus generates an analytic semigroup
  \begin{align}\label{EQ74}
    S_{\xi}(t) := \int_{\Gamma} e^{\lambda t} R(\lambda; A_{\xi}) d \lambda,
  \end{align}
where $\Gamma \subseteq \rho(A_{\xi})$ is a path which starts
as $\re \lambda \to -\infty$ as $\im \lambda \to -\infty$,
encircles $0$ on the right side and ends
as $\re \lambda \to -\infty$ and $\im \lambda \to \infty$.

\cole
\begin{theorem}\label{T1}
The operator $A_{\xi}$ is sectorial and generates an analytic semigroup $S_{\xi}$ on~$L^2$. The resolvent set $\rho(A_{\xi})$ of $A_{\xi}$ contains $\mathbb{C} \setminus \left((-\infty , -\nu |\xi|^2] \cup \{0\}\right)$. Moreover, we have
  \begin{align}\label{EQ75}
  \|R(\lambda; A_{\xi})\|_{L^2\to L^2}  \leq C_{\theta}{\left|\lambda + \nu |\xi|^2\right|^{-1}} 
  \end{align}
  and
  \begin{align}
    \|R(\lambda; A_{\xi})\|_{L^2\to H^1}  \leq C_{\theta}\nu^{-1/2}|\lambda + \nu|\xi|^2 |^{-1/2} \label{EQ76}
  \end{align}
for
  \begin{equation}
   \lambda \in S_{\theta, \xi}= \{\lambda' \in \mathbb{C}: |\lambda'| \ge c_0|\xi|^2,\ |\arg\lambda'|<\theta\}
   ,
   \llabel{EQ232}
  \end{equation}
where $\theta\in(0,\pi)$ is a constant and $c_0>0$ is arbitrarily small.
\end{theorem}
\colb

\begin{proof}[Proof of Theorem~\ref{T1}]
With $\lambda u - A_{\xi} u=f$, we aim to bound $\|u\|_{L^2}$ by $\|f\|_{L^2}$. Recall that $u(y)=v(y) + w(y)$, where
$v$ and $w$ are given in \eqref{EQ58} and \eqref{EQ68}, respectively.
By \eqref{EQ63} and Plancherel's identity, we have
  \begin{equation}\label{EQ77}
   \|v\|_{L^2}^2 =\| \mathcal{F}(v)\|_{L^2}^2
   =c\int_0^{+\infty} \frac{1}{ |\lambda+\nu( |\xi|^2+ |\zeta|^2)|^2}|\mathcal{F}(\bar{f})(\zeta)|^2 d\zeta
   \lec \frac{1}{|\lambda+\nu |\xi|^2|^2} \|f\|_{L^2}^2
   ,
  \end{equation}
where $c$ is a constant.
For $w$, note that
  \begin{equation}
   B^{-1}=\frac{1}{|\xi| \mu (\mu-|\xi|)}  \begin{pmatrix}
   \mu|\xi|- \xi_1^2 & -\xi_1\xi_2 \\
   -\xi_1\xi_2  & \mu |\xi|- \xi_2^2 \\
   \end{pmatrix}
   \llabel{EQ81}
  \end{equation}
and
  \begin{equation}
   (\partial_{3}+|\xi|)v -  \xi|\xi|^{-1}\xi\cdot v
     = |\xi|v -  \xi|\xi|^{-1}\xi\cdot v
     \onon{\partial \mathbb{R}_{+}}
    .
   \llabel{EQ82}
  \end{equation}
Hence, we have
  \begin{equation}
   w(z)
   = \frac{1}{|\xi|(\mu-|\xi|)} \begin{pmatrix} \xi_2\\-\xi_1 \end{pmatrix} \gamma(\xi_2 v_1 -\xi_1 v_2)e^{-\mu z}
   .
   \llabel{EQ83}
  \end{equation}
Since also
  \begin{equation}
   |v(0)| \lec \frac{1}{\nu \mu\sqrt{\mu}} \|f\|_{L^2},
   \llabel{EQ84}
  \end{equation}
we get
  \begin{equation}
   \|w\|_{L^2}
   \lec
   \frac{|\xi|}{|\sqrt{\mu}(\mu-|\xi|)|}  |v(0)|
   \lec
   \frac{1}{\nu \mu^2}\frac{|\xi|}{|\mu-|\xi||}  \|f\|_{L^2}
   ,
   \llabel{EQ230}
  \end{equation}
where the norms are taken in the $z$~variable.
Recalling the definition of $\mu$ from \eqref{EQ80}, we obtain 
  \begin{equation}
   \|w\|_{L^2}
   \lec
   \frac{|\xi|}{|\lambda + \nu |\xi|^2|\,|\mu-|\xi||}
   \|f\|_{L^2}
   \leq
   \frac{1}{(\sqrt{1+ \fractext{\lambda}{\nu |\xi|^2}}-1)|\lambda + \nu |\xi|^2|}
   .
   \label{EQ231}
   \end{equation}
Since $\lambda \in S_{\theta, \xi}$, there exists a constant $C_{\theta,\omega}$ such that
  \begin{equation}
   \|w\|_{L^2}\leq \frac{C_{\theta,\omega}}{|\lambda + \nu |\xi|^2|} \|f\|_{L^2}.
  \label{EQ78}
  \end{equation}
Combining \eqref{EQ77} and \eqref{EQ78}, we obtain~\eqref{EQ75}.
Next, by Young's inequality, we have 
  \begin{align}
  \begin{split}
    \|v\|_{H^1}^2
    &=\|\zeta \mathcal{F}(v)\|_{L^2}^2=\frac{1}{\nu}\int_0^{+\infty} \frac{\nu |\zeta|^2}{ |\lambda+\nu( |\xi|^2+ |\zeta|^2)|^2}|\mathcal{F}(\bar{f})(\zeta)|^2\,d\zeta
    \\&
    \leq
    \frac{1}{\nu|\lambda+\nu |\xi|^2|} \|f\|_{L^2}^2.
  \end{split}
  \label{EQ79}
  \end{align}
Similarly, 
  \begin{align}
  \begin{split}
   \|w\|_{H^{1}}
   &
   \leq
   \frac{\sqrt{\mu}|\xi|}{|(\mu-|\xi|)|}  |v(0)|
   \lec \frac{|\xi|}{\nu\mu|\mu-|\xi||}
   \|f\|_{L^2}
   \lec
   \frac{1}{(\sqrt{1+ \fractext{\lambda}{\nu |\xi|^2}}-1)|\lambda + \nu |\xi|^2|^{1/2}\sqrt{\nu}}
   \|f\|_{L^2}.
  \end{split}
  \label{EQ79.1}
  \end{align}
Combining \eqref{EQ79} and \eqref{EQ79.1}, we then obtain~\eqref{EQ76}.
\end{proof}

\subsection{Extension of $G_{\lambda,\xi}$ and $S_{\xi}(t)$}\label{secGreen}
We now extend our definition of
$R(\lambda, A_\xi)$ and $S_{\xi}(t)$ from $L^2(\mathbb{R}_{+})$ to the space of
finite Borel measures~$\mathcal{M}=\mathcal{M}(\mathbb{R}_{+})$.  We first prove the following lemma.

\cole
\begin{lemma}
\label{L04}
The operator $R(\lambda; A_{\xi})$ is symmetric for all $\lambda \in \rho(A_{\xi})$.
\end{lemma}
\colb

\begin{proof}[Proof of Lemma~\ref{L04}]
For any $u, v\in  D(A_{\xi})\subseteq C^{1}(\mathbb{R}_+)$,
we denote
$f = (\lambda-A_{\xi}) u$ and $g= (\lambda-A_{\xi})v$.
Integrating by parts gives
  \begin{equation}
  \begin{aligned}
  \langle f, R(\lambda;A_{\xi})g \rangle_{L^2}
   &=
     \langle (\lambda- A_{\xi})u , v\rangle_{L^2}
   \\&
   =\langle u , ({\lambda}+ \nu |\xi|^2) v\rangle_{L^2} - \nu \partial_z u(z) \cdot \bar{v}(z)|_{z=0} 
   \\&\indeq
   + \nu u(z) \cdot \partial_{z} \bar{v}(z)|_{z=0} 
   - \int_0^{+\infty} u(z) \cdot \nu \partial_z^2  \bar{v}(z) \,dz.
  \end{aligned}
  \llabel{EQ87}
  \end{equation}
Since $u, v \in D(A_{\xi})$, applying the boundary condition gives
  \begin{align}
    \begin{split}
   \langle (\lambda- A_{\xi})u , v\rangle_{L^2}
   &= \langle u , ({\lambda}+ |\xi|^2) v\rangle_{L^2} 
     +   u \cdot \nu (|\xi| \bar{v} + \partial_z \bar{v} - \xi |\xi|^{-1} \xi \cdot \bar{v} )|_{z=0}
     \\&\indeq
  - \int_0^{+\infty} u(z) \cdot \nu \partial_z^2 \bar{v}(z) dz,
  \end{split}
  \llabel{EQ85}
  \end{align}
from where it follows
\begin{align}
  \begin{split}
\langle f, R(\lambda;A_{\xi}) g  \rangle_{L^2}=\langle (\lambda- A_{\xi})u , v\rangle_{L^2} = \langle  u, (\lambda - A_{\xi})v   \rangle_{L^2}  = \langle R(\lambda;A_{\xi}) f ,g \rangle_{L^2}
  ,
  \end{split}
   \llabel{EQ86}
\end{align}
and the proof is complete.
\end{proof}

For $\phi \in \mathcal{M}$, we consider the linear functional $T_{\phi}\colon L^2(\mathbb{R}_+) \to \mathbb{R}$ defined by
  \begin{equation}
   T_{\phi}(f) = \langle \phi ,  R({\lambda} ; A_{\xi}) f \rangle
  \comma
  f\in L^{2}(\mathbb{R}_{+})
   \llabel{EQ88}
 \end{equation}
Since $R({\lambda} ; A_{\xi}) f \in H^1(\mathbb{R}_+) \subseteq  C_{\text{b}}(\mathbb{R}_+)$
for $f\in L^{2}(\mathbb{R}_{+})$,
we obtain by~\eqref{EQ76}, that
$T_{\phi}$ is a bounded linear functional on
$L^{2}(\mathbb{R}_{+})$.
Therefore,
for $\phi\in\mathcal{M}$,
we may define $R(\lambda; A_{\xi})\phi$ as an element in $L^2(\mathbb{R}_+)$.

\begin{definition}\label{EQ89}
For any $\phi \in \mathcal{M}$, we define
$R(\lambda; A_{\xi}) \phi$ to be an element in $L^2(\mathbb{R}_+)$ such that
  \begin{align}
   \langle R(\lambda ; A_{\xi})\phi, f  \rangle_{L^2(\mathbb{R}_+)}= \langle \phi ,  R({\lambda} ; A_{\xi}) f \rangle_{\mathcal{M} \times C_{\text{b}}(\mathbb{R}_{+})} ,
   \llabel{EQ90}
  \end{align}
for all $f \in L^2(\mathbb{R}_+)$ and~$\lambda \in \rho(A_{\xi})$.
\end{definition}

\medskip
Since $S_{\xi}(t)$ is an analytic semigroup, we have $S_{\xi}(t) f \in D(A_{\xi})$ for $f \in L^2(\mathbb{R}_+)$ and $t>0$, and hence  $S_{\xi}(t)f \in L^{\infty}(\mathbb{R}_+)$. Therefore,
for $\phi\in\mathcal{M}$,
we may similarly define $S_{\xi}(t)\phi$ as an element in $L^2(\mathbb{R}_+)$ as follows.

\begin{definition}\label{semigroup distribution}
For any $\phi\in \mathcal{M}$
and $t>0$, we define $S_{\xi}(t) \phi$ as an element in $L^2(\mathbb{R}_+)$ such that
\begin{align}
  \begin{split}
  \langle S_{\xi}(t) \phi , f\rangle_{L^2(\mathbb{R}_+)} = \langle \phi, S_{\xi}(t) f \rangle_{\mathcal{M} \times C_{\text{b}}(\mathbb{R}_{+})} .
  \end{split}
  \llabel{EQ92}
  \end{align}
for any $f \in L^2(\mathbb{R}_+)$.
\end{definition}

\begin{remark}
Due to analyticity of the semigroup, we immediately obtain that
$S_\xi(t) f \in C^\infty(\mathbb{R}_{+})$ for $t>0$ and~$f\in L^2$. 
\end{remark}

\startnewsection{Green's Function}{sec04}

\subsection{Analytic representation of Green's function}
Now, we introduce Green's function
  \begin{equation}
   G_{\lambda,\xi}(z):= R(\lambda; A_{\xi})\diag{(\d_z, \d_z})
   .
   \label{EQ233}
  \end{equation}
From Definition~\ref{EQ89}, we have
  \begin{align}
    \begin{split}
  \langle G_{\lambda, \xi}(y), f \rangle_{} = \langle \diag{(\d_y, \d_y)},  R(\lambda;A_{\xi}) f \rangle
      \comma f \in (L^{2}(\mathbb{R}_+))^2
   ,
  \end{split}
   \llabel{EQ91}
  \end{align}
where the duality $\mathcal{M}$-$C_{\text{b}}(\mathbb{R}_{+})$ is understood.
By \eqref{EQ58} and \eqref{EQ68}, we have
  \begin{equation}
   \langle G_{\lambda, \xi}(y), f \rangle
   =
    \frac{1}{2 \nu \mu}
   \int_{0}^{\infty} (e^{-\mu |y-z|} + e^{-\mu|y+z|}) f(z) dz + \int_0^{\infty} M(z) f(z) e^{-\mu y}\,dz
   \label{EQ93}
     \end{equation}
where $M(z) \in \mathbb{R}^{2 \times 2}$ solves the equation
\begin{align}
  \begin{split}
   BM f(y)
   &= \frac{1}{2 \nu \mu}(\partial_{z}+|\xi|)(e^{-\mu |y-z|} + e^{-\mu |y+z|})f |_{z=0}
   \\&\indeq
   -   \frac{1}{2 \nu \mu}\xi|\xi|^{-1}\xi\cdot (e^{-\mu |y-z|} + e^{-\mu |y+z|})f |_{z=0} ,
  \end{split}
   \llabel{EQ94}
     \end{align}
and $B$ is given in \eqref{EQ73}.
Hence, we obtain 
  \begin{align}
  \begin{pmatrix}
  \mu-|\xi| + \xi_1^2 |\xi|^{-1} & \xi_1\xi_2 |\xi|^{-1}\\
  \xi_1\xi_2 |\xi|^{-1} & \mu -|\xi| + \xi_2^2 |\xi|^{-1}\\
  \end{pmatrix}
  M(z) f
  = \frac{e^{-\mu z}}{\nu\mu}
  \begin{pmatrix}
  (\xi_2^2 f_1 - \xi_1 \xi_2 f_2)|\xi|^{-1}\\
  (-\xi_1\xi_2 f_1+ \xi_1^2 f_2)|\xi|^{-1}
  \end{pmatrix}
  \llabel{EQ96}
  \end{align}
for any $f\in \mathbb{R}^2$,
which implies that
  \begin{align}
  \begin{split}
  \begin{pmatrix}
  \mu |\xi| - \xi_2^2  &  \xi_1\xi_2 \\
  \xi_1\xi_2  & \mu|\xi| - \xi_1^2 \\
  \end{pmatrix}
  M(z)
  = \frac{e^{-\mu z}}{\nu\mu}
  \begin{pmatrix}
      \xi_2^2 &  -\xi_1 \xi_2\\
  -\xi_1\xi_2 & \xi_1^2 
  \end{pmatrix}.
  \end{split}
   \label{EQ97}
   \end{align}
Computing $M(z)$ from \eqref{EQ97} then leads to
  \begin{align}
  \begin{split}
   M(z)
   &=
   \frac{e^{-\mu z}}{\nu\mu}    
    \frac{1}{\mu|\xi|^2(\mu-|\xi|)}
  \begin{pmatrix}
  \mu |\xi|-\xi_1^2 & - \xi_1\xi_2 \\ - \xi_1\xi_2 & \mu|\xi|-\xi_2^2
  \end{pmatrix}
  \begin{pmatrix}
  \xi_2^2 & - \xi_1\xi_2 \\ - \xi_1\xi_2 & \xi_1^2 
  \end{pmatrix}
    \\&
   =
   \frac{e^{-\mu z}}{\nu\mu} 
    \frac{1}{|\xi|(\mu-|\xi|)}
  \begin{pmatrix}
  \xi_2^2 & - \xi_1\xi_2 \\ - \xi_1\xi_2 & \xi_1^2 
  \end{pmatrix}
   ,
  \end{split}
   \llabel{EQ235}
  \end{align}
and replacing this in \eqref{EQ93}, we get
  \begin{align}
    \begin{split}
      \langle G_{\lambda, \xi}(y), f \rangle &=  \int_{0}^{\infty}\frac{1}{2 \nu \mu} (e^{-\mu |y-z|} + e^{-\mu|y+z|})   \begin{pmatrix}f_1(z)\\ f_2(z) \end{pmatrix} dz
    \\&\indeq
    + \int_0^{\infty} \frac{1}{\nu \mu} e^{-\mu|y+z|}
    \frac{1}{|\xi|(\mu-|\xi|)}
  \begin{pmatrix}
  \xi_2^2 & - \xi_1\xi_2 \\ - \xi_1\xi_2 & \xi_1^2 
  \end{pmatrix} 
  \begin{pmatrix}
  f_1(z)\\
  f_2(z)
  \end{pmatrix}
  \,dz
  .
  \end{split}
   \llabel{EQ98}
  \end{align}
Finally, we note that
\begin{equation}
\frac{1}{\nu \mu|\xi|(\mu-|\xi|)} = \frac{\mu+ |\xi|}{ \mu |\xi| (\nu \mu^2- \nu|\xi|^2)}= \frac{\mu+|\xi|}{\mu |\xi|\lambda }
   .
   \llabel{EQ99}
     \end{equation}
Thus we obtain the following result.

\cole
\begin{lemma}\label{l concrete resolvent}
Let $\mu= \nu^{-1/2} \sqrt{\lambda + |\xi|^2 \nu}$ where
$\lambda \in \mathbb{C}\setminus(-\infty, -|\xi|^2 \nu]$
and $\lambda\neq0$. There holds
  \begin{align}
    \begin{split}
  \langle G_{\lambda, \xi}(y), f \rangle_{} = \int_0^{\infty}G_{\lambda,\xi} (y,z) f(z) \,dz
      \comma f\in C_{0}^{\infty}(\mathbb{R}_{+})
  ,
  \end{split}
   \llabel{EQ100}
  \end{align}
and the kernel $G_{\lambda,\xi}(z,y)$ is given by 
\begin{equation}
G_{\lambda,\xi} (y,z) = H_{\lambda,\xi} (y,z) + R_{\lambda,\xi}(y,z)
   ,
   \llabel{EQ101}
     \end{equation}
where
  \begin{equation}
  H_{\lambda, \xi}(y,z) =  \frac{1}{2 \nu \mu} (e^{-\mu |y-z|} + e^{-\mu|y+z|}) I_2
  \llabel{EQ102}
  \end{equation}
and
  \begin{equation}
   R_{\lambda, \xi}(z,y)=\frac{\mu+ |\xi|}{ \mu  \lambda |\xi|}
  \begin{pmatrix}
  \xi_2^2 & -\xi_1\xi_2\\
  -\xi_1\xi_2 & \xi_1^2
  \end{pmatrix}
  e^{-\mu|y+z|}
   ,
   \label{EQ236}
  \end{equation}
with $I_2$ denoting the $2\times2$ identity matrix.
\end{lemma}
\colb

Next, we calculate the time-dependent
kernel $G_{\xi}(t,y)= S_{\xi}(t)\d_y$, where
  $t>0$. By Definition \ref{semigroup distribution}, we have
  \begin{align}
  \begin{split}
\langle S_{\xi}(t)\d_y , f \rangle_{}=& \langle \d_y , S_{\xi}(t) f \rangle   \comma f \in L^2(\mathbb{R}_+).
  \end{split}
   \llabel{EQ103}
  \end{align}
Using Definitions~\ref{EQ89} and~\ref{semigroup distribution}, we obtain
  \begin{equation}
  \begin{aligned}
&\langle G_{\xi}(t,y) , f \rangle= \langle S_{\xi}(t) \d_y, f \rangle=  \langle  \d_y, S_{\xi}(t) f \rangle
=  \int_{\Gamma} e^{\lambda t} \langle  \d_y , R(\lambda; A_{\xi}) f  \rangle_{}\,d \lambda
= \int_{\Gamma} e^{\lambda t} \langle   R(\lambda; A_{\xi})\d_y ,  f  \rangle_{}\,d \lambda,
\end{aligned}
   \llabel{EQ104}
     \end{equation}
where $\Gamma$ is a path as in after~\eqref{EQ74}.
By Lemma~\ref{l concrete resolvent} and Fubini's Theorem, we arrive at
\begin{align}
  \begin{split}
   \langle G_{\xi}(t,y) , f \rangle =& \int_{\Gamma} e^{\lambda t}  \int_0^{\infty}G_{\lambda,\xi} (y,z)  f(z) \,dz  \,d \lambda
  =
   \int_0^{\infty}
   \left(\int_{\Gamma} e^{\lambda t } G_{\lambda,\xi} (y,z)  \,d\lambda\right) f(z)   \,dz.
  \end{split}
   \llabel{EQ105}
 \end{align}
Defining
  \begin{equation}\label{eq green}
   G_{\xi}(t,y;z):=  \int_{\Gamma} e^{\lambda t } G_{\lambda,\xi} (y,z)\,d\lambda,
  \end{equation}
we have the following theorem.

\cole
\begin{theorem}\label{T2}
The kernel $G_{\xi}(t,y;z)$ has the  form
  \begin{align}
   G_{\xi}(t,y;z)= H_{\xi}(t,y;z) + R_{\xi}(t,y;z),
   \label{EQ106}
  \end{align}
where
  \begin{align}\label{EQ107}
    H_\xi(t, y ,z) = \frac{1}{\sqrt{\nu t}}\left(e^{-\frac{|z-y|^2}{4\nu t}} + e^{-\frac{|z+y|^2}{4\nu t}} \right)e^{-\nu|\xi|^2t} I_{2}
   \end{align}
and
\begin{align}\label{EQ108}
R_{\xi}(t,y;z)= \int_{\Gamma} e^{\lambda t} \frac{\mu+ |\xi|}{ \mu  \lambda |\xi|} 
\begin{pmatrix}
\xi_2^2 & -\xi_1\xi_2\\
-\xi_1\xi_2 & \xi_1^2
\end{pmatrix}
 e^{-\mu|y+z|}\,d\lambda,
\end{align}
with $\mu= \nu^{-1/2} \sqrt{\lambda + |\xi|^2 \nu}$
and $I_2$ denoting the $2\times2$ identity matrix.
The matrix function \eqref{EQ106} solves the system
\begin{subequations}
  \begin{align}
    \partial_{t}G_{\xi}(t,y; z) - \nu\Delta_\xi G_{\xi}(t,y; z)&= 0\label{EQ109}
    \commaone t>0\\
    -\nu(\partial_{z}+|\xi|)G_{\xi}(t,y;z)|_{z=0} + \nu \xi|\xi|^{-1}\xi\cdot G_{\xi}(t,y;z)|_{z=0}&=0\label{EQ110},\\
    G_{\xi}(t,y;z)|_{t=0} &{= \d_y(z)I_2}
    \comma z>0
    ,
 \label{EQ111}
\end{align}
for~$y>0$.
\end{subequations}
\end{theorem}
\colb

\begin{proof}[Proof of Theorem~\ref{T2}]
We get the equation~\eqref{EQ107} by solving the heat equation with the homogeneous Neumann condition, while \eqref{EQ108} follows directly from~\eqref{EQ236} and \eqref{eq green}. By Lemma~\ref{l concrete resolvent}, both $G_{\xi}$ and $H_{\xi}$ are continuous with respect to $z$, and \eqref{EQ110} holds. To prove \eqref{EQ109}, it suffices to check that $\langle \partial_t G_{\xi}(t,y) - A_{\xi} G_{\xi}(t,y), f  \rangle=0$, for any $f \in C_0^{\infty}(\mathbb{R}_+)$.
Indeed, integration by parts gives
\begin{align}
  \begin{split}
  \langle A_{\xi} G_{\xi}(t,y),f  \rangle_{}= \int_0^{\infty} (-\xi^2 + \partial_z^2) G_{\xi}(t,y;z) f(z) \,dz
  =  \int_0^{\infty}  G_{\xi}(t,y;z)(-\xi^2 + \partial_z^2) f(z) \,dz.
  \end{split}
   \llabel{EQ115}
  \end{align}
Hence, we have
\begin{equation}
\langle \partial_t G_{\xi}(t,y)- A_{\xi} G_{\xi}(t,y),f  \rangle= \langle \partial_t S_{\xi}(t)\d_y , f  \rangle_{} -\langle   G_{\xi}(t,y), \Delta_{\xi} f  \rangle
   .
   \llabel{EQ116}
     \end{equation}
By the absolute convergence of the integral
  \begin{equation}
  \int_0^{\infty} \int_{\Gamma} \lambda e^{\lambda t } G_{\lambda,\xi} (y,z)  \,d\lambda f(z)  \,dz,
   \llabel{EQ117}
     \end{equation}
we have
\begin{equation}
 \langle \partial_t S_{\xi}(t)\d_y , f  \rangle = \langle  \d_y , \partial_t S_{\xi}(t) f  \rangle
   .
   \llabel{EQ118}
     \end{equation}
Hence, we obtain 
  \begin{align}
    \begin{split}
  \langle \partial_t G_{\xi}(t,y) - A_{\xi} G_{\xi}(t,y),f  \rangle=& \langle  \d_y , \partial_t S_{\xi}(t) f  \rangle- \langle  S_{\xi}(t) \d_y, A_{\xi}f  \rangle\\
=& \langle  \d_y , \partial_t S_{\xi}(t) f  \rangle - \langle   \d_y, A_{\xi}S_{\xi}(t) f  \rangle.
  \end{split}
   \llabel{EQ119}
  \end{align}
By 
the analytic semigroup property of $S_{\xi}$ in $L^2$, there holds
\begin{equation}
\partial_t S_{\xi}(t)f -A_{\xi} S_{\xi}(t)f=0
,
   \llabel{EQ120}
     \end{equation}
from where
\begin{equation} 
\langle  \d_y , \partial_t S_{\xi}(t) f - A_{\xi}S_{\xi}(t) f  \rangle=0,
   \llabel{EQ121}
     \end{equation}
establishing \eqref{EQ109}.

The boundary condition \eqref{EQ110} follows by a direct computation, and we omit further details. 

Finally, we check~\eqref{EQ111}. For $f \in C_0^{\infty}(\mathbb{R}_+)$, we have
\begin{equation}
\langle S_{\xi}(t) \d_y , f \rangle := \langle \d_y , S_{\xi}(t)  f \rangle
\aand
\lim_{t\to 0} S_{\xi}(t)f =f
,
   \llabel{EQ242}
\end{equation}
by Definition \ref{semigroup distribution},
where the convergence holds uniformly on compact sets.
Therefore, we have
\begin{equation}
\lim_{t \to 0}  \langle S_{\xi}(t) \d_y , f \rangle = f(y),
   \llabel{EQ243}
\end{equation}
which implies \eqref{EQ111}.
\end{proof}

\subsection{An estimate for Green's function}
In this section, we prove a result of Nguyen-Nguyen type, providing an upper bound for Green's function.

\cole
\begin{theorem}\label{T3}
The residual kernel $R_{\xi}$
may be decomposed as 
  \begin{equation}
    R_{\xi}= R^{(1)}_{\xi}+ R^{(2)}_{\xi}   
   \label{EQ112}
  \end{equation}
where $R^{(1)}$ and $R^{(2)}$ are defined in 
\eqref{EQ240} and \eqref{EQ241} below,
with the two kernels satisfying the bounds
  \begin{equation}
    |\partial_{z}^{k}R^{(1)}_\xi(t, y, z)|
    \lesssim \mu_0^{k+1} e^{-\theta_0\mu_0(y+z)} 
    \comma k\in{\mathbb N}_0
    \label{EQ128}
  \end{equation}
  and
  \begin{equation}
    |\partial_{z}^{k} R^{(2)}_\xi(t, y,z) |
    \lesssim \frac{1}{(\nu t)^{(k+1)/2}}
    e^{-\frac{(y+z)^2}{\nu t}}e^{-\frac{\nu|\xi|^2t}{8}}
    \comma k\in{\mathbb N}_0
    \label{EQ129}
  \end{equation}
  where 
  $\theta_0>0$ is a constant and
  the boundary remainder coefficient 
  is given by $\mu_0=\mu_0(\xi, \nu)= |\xi|+\sqrt{\nu}^{-1}$.
The implicit constants in \eqref{EQ128} and \eqref{EQ129} depend on~$k$.
\end{theorem}
\colb

\begin{proof}[Proof of Theorem~\ref{T3}]
We first consider the case when $\nu |\xi|^2 \leq 1 $ and,
as in \cite{NN},
decompose the contour as $\Gamma:= \Gamma_{-}\cup \Gammac\cup \Gamma_{+}$, where
  \begin{align}
  \begin{split}
   &
   \Gamma_{-}:= \left\{ \lambda= -\frac{1}{2}\nu |\xi|^2 + \nu (a^2-b^2) + 2 \nu i ab - i M,\ b \in (-\infty,0]  \right\},
   \\&
    \Gammac:= \left\{ \lambda =-\frac{1}{2}\nu|\xi|^2 + \nu a^2 + Me^{i\theta}, \  \theta\in [-\pi/2, \pi/2]  \right\},
\\&
      \Gamma_{+}:= \left\{ \lambda= -\frac{1}{2}\nu |\xi|^2 + \nu (a^2-b^2) + 2 \nu i a b +  i M,\  b \in [0,\infty)  \right\},
  \end{split}
   \llabel{EQ237}
  \end{align}
where $a=\fractext{|y+z|}{2\nu t}$
and $M>0$ is large enough so that $0$ is on the left-hand side of the contour $\Gamma$.

We first consider the integral on~$\Gammac $.
Recalling \eqref{EQ80},
we have
\begin{equation}
\re \mu \geq \sqrt{\frac{\re \lambda}{\nu} + |\xi|^2}\geq \frac{\sqrt{\re \lambda}}{2}\left(|\xi| + \frac{1}{\sqrt{\nu}}\right)
   \llabel{EQ132}
     \end{equation}
and
  \begin{equation}
   \re \mu \geq \sqrt{\re \lambda /\nu} \geq a   
   .
   \llabel{EQ238}
  \end{equation}
Since $|\lambda |\gtrsim 1$ on $\Gammac $,  there exists a constant $\theta_0$, depending on the choice of $\Gammac$, such that
\begin{equation}
\re \mu \geq \theta_0 \mu_0 + \frac{a}{2}.
   \llabel{EQ133}
     \end{equation}
For the case $\nu|\xi|^{2}\leq 1$, we set
  \begin{equation}
   R^{(1)}_\xi(t, y, z)
   =
   \int_{\Gammac} e^{\lambda t} R_{\xi}( \lambda ,y,z) \,d\lambda
   \label{EQ240}
  \end{equation}
and
  \begin{equation}
   R^{(2)}_{\xi}
   =
   \int_{\Gamma_{-}\cup\Gamma_{+}}
   e^{\lambda t} R_{\xi}( \lambda ,y,z) \,d\lambda
   .
   \label{EQ241}
  \end{equation}
Then we have
  \begin{equation}
   \left| R^{(1)}_\xi(t, y, z)\right|
   \lesssim \int_{\Gammac}  e^{Mt} e^{\nu a^2 t}e^{-\frac{a}{2}|y+z|} e^{-\theta_0 \mu_0 |y+z|}\left| \frac{\mu+ |\xi|}{ \mu  }\right| |\xi| \,|d\lambda|
  .
   \llabel{EQ134}
     \end{equation}
Recalling that $\mu= \nu^{-1/2} \sqrt{\lambda + |\xi|^2 \nu}$ from \eqref{EQ80}, we bound
  \begin{equation}
   \left|\frac{\mu+ |\xi|}{ \mu  } |\xi|\right| \leq |\xi| + \frac{|\xi|^2}{|\mu|}
    = |\xi| + |\xi| \left|\sqrt{\frac{\nu |\xi|^2}{\lambda + |\xi|^2 \nu}}\right| \lesssim \mu_0
   ,
   \llabel{EQ135}
     \end{equation}
obtaining
\begin{equation}
\left|\int_{\Gammac} e^{\lambda t} R_{\xi}(t,y,z) \,d\lambda \right |  \lesssim \mu_0 e^{\theta_0 \mu_0 |y+z|}
   .
   \llabel{EQ136}
     \end{equation}
For the integrals on $\Gamma_{-}$ and $\Gamma_{+}$, we have
\begin{equation}
  \left| e^{\lambda t} e^{-\mu|y+z|} \right|
  \leq e^{-\frac{1}{2}\nu |\xi|^2 t} e^{-\frac{|y+z|^2}{4\nu t}} e^{-\nu b^2 t}
    \comma \lambda\in \Gamma_{-} \cup \Gamma_{+}
    ,
   \llabel{EQ137}
     \end{equation}
which gives 
  \begin{align}
    \begin{split}
   \left| \int_{\Gamma_{-}\cup \Gamma_{+}} e^{\lambda t} R_{\xi}(\lambda,y,z) \,d\lambda \right | \leq  \int_{\Gamma_{-}\cup \Gamma_{+}}  e^{-\frac{1}{2}\nu |\xi|^2 t} e^{-\frac{|y+z|^2}{4\nu t}} e^{-\nu b^2 t} \left|\frac{\mu+ |\xi|}{ \mu  \lambda} |\xi| \right| |d \lambda|
   .
   \end{split}
   \llabel{EQ138}
\end{align}
Since on $\Gamma_{-}\cup\Gamma_{+}$, we have $\lambda= -\frac{1}{2}\nu |\xi|^2 + \nu (a^2-b^2) + 2 \nu i ab \pm i M$, we obtain
  \begin{equation}
   d \lambda = -2\nu i (a+ bi) \,db
   ,
   \llabel{EQ139}
     \end{equation}
and thus
  \begin{equation}
 \frac{\mu+ |\xi|}{ \mu  \lambda } \,d\lambda= \pm \frac{2i (a+bi)}{\mu(\mu- |\xi|)} \,db.
   \llabel{EQ140}
     \end{equation}
Since $\mu^2= \lambda/\nu + |\xi|^2$ on $\Gamma_{-}\cup\Gamma_{+}$, we get
\begin{equation}
\mu^2 = \frac{1}{2} |\xi|^2 + (a+ ib)^2 \pm i M \nu^{-1}
   ,
   \llabel{EQ141}
     \end{equation}
which implies 
\begin{equation}
|a+ib|^2 \lesssim |\mu|^2 + |\xi|^2 + \nu^{-1}  \lesssim |\mu|^2 
   .
   \llabel{EQ142}
     \end{equation}
Therefore, we obtain that
\begin{align}
  \begin{split}
\left| \int_{\Gamma_{-}\cup\Gamma_{+}} e^{\lambda t} R_{\xi}(\lambda ,y,z) \,d\lambda \right| \lesssim&  \int_{\mathbb{R}}  e^{-\frac{1}{2}\nu |\xi|^2 t} e^{-\frac{|y+z|^2}{4\nu t}} e^{-\nu b^2 t} \frac{|\xi|\cdot|a+bi|}{|\mu(\mu-|\xi|)|} d b\\
\lesssim&  \int_{\mathbb{R}}  e^{-\frac{1}{2}\nu |\xi|^2 t} e^{-\frac{|y+z|^2}{4\nu t}} e^{-\nu b^2 t} \frac{|\xi|\cdot|\mu|}{|\mu (\mu-|\xi|)|} d b
  .
  \end{split}
   \llabel{EQ143}
\end{align}
Let $\nu |\xi|^2 = r^{-1}$, and note that $1 \le r$. Since $ 1 \lesssim |\lambda|$ on $\Gamma_{-}\cup\Gamma_{+}$, we may write
\begin{equation}
\frac{|\xi|}{|\mu-|\xi||} \leq \frac{1}{|\sqrt{1+ r \lambda}-1|} \lesssim \frac{1+ \sqrt{r |\lambda|}}{r|\lambda|} \lesssim 1.
   \llabel{EQ144}
     \end{equation}
Hence,
\begin{align}
\left| \int_{\Gamma_{-}\cup\Gamma_{+}} e^{\lambda t} R_{\xi}(\lambda ,y,z) \,d\lambda \right| \lesssim&  \int_{\Gamma_{-}\cup\Gamma_{+}}  e^{-\frac{1}{2}\nu |\xi|^2 t} e^{-\frac{|y+z|^2}{4\nu t}} e^{-\nu b^2 t}   d b \lesssim (\nu t)^{-1/2}  e^{-\frac{1}{2}\nu |\xi|^2 t} e^{-\frac{|y+z|^2}{4\nu t}}.
   \llabel{EQ145}
   \end{align}
For the derivatives, we have
\begin{equation}
\left|\partial_z^k \int_{\Gammac} e^{\lambda t} R_{\xi}(\lambda ,y,z) \,d\lambda \right|  \lesssim \int_{\Gammac} \mu^k e^{-\theta_0 \mu_0 |y+z|} \left| \frac{\mu+ |\xi|}{ \mu  }\right| |\xi| \,|d\lambda|
   .
   \llabel{EQ146}
   \end{equation}
Since $\mu \lesssim \nu^{-1/2}\lesssim \mu_0$, under the condition $|\xi|^2 \nu \leq 1$, we may bound
   \begin{equation}
   \left|\int_{\Gammac} e^{\lambda t} R_{\xi}(\lambda ,y,z) \,d\lambda \right|  \leq \mu_0^{k+1} e^{\theta_0 \mu_0 |y+z|}.
   \llabel{EQ147}
   \end{equation}
The derivative for the part of $\Gamma_{-}\cup\Gamma_{+}$ can be treated similarly as for~$\Gammac$.

We now address the case  $\nu |\xi|^2 \geq 1$, considering the contour 
  \begin{equation}
   \Gamma:= \left\{\lambda=  -\nu|\xi|^2 + \nu (\theta^2 a^2-b^2) + 2\nu i \theta a b,\ b \in \mathbb{R} \right\},
   \llabel{EQ148}
   \end{equation}
where $\theta=1$ when $a/|\xi|\in (-\infty, 1/2)\cup (3/2, \infty)$ and $\theta=1/2$ when $a/|\xi| \in (1/2,3/2)$. 
Note that we do not require that $0$ is on the left of the contour any more; we have
\begin{align}
  \begin{split}
  R_{\xi}(t,y;z) := R_\xi^{(1)}(t, y; z) + R_\xi^{(2)}(t, y ;z)
  = R_\xi^{(1)}(t, y; z) + \int_{\Gamma} e^{\lambda t}e^{-\mu |y+z|} \frac{\mu + |\xi|}{\mu \lambda} |\xi| d \lambda  
  ,
  \end{split}
   \llabel{EQ149}
  \end{align}
where $R_\xi^{(1)}(t, y; z)$, for the case
$\nu|\xi|^{2}> 1$,
is
defined as
the residue of $e^{\lambda t} R_{\xi}(\lambda ,y,z)$ at the pole $\lambda=0$ if $0$ is on the right of the contour $\Gamma$ and $0$ otherwise.
On $\Gamma$, we may write
  \begin{equation}
   \mu = \theta a + ib
   \aand
   \lambda = \nu (\theta a +ib + |\xi|)  (\theta a +ib - |\xi|) 
   \llabel{EQ150}
   \end{equation}
and
\begin{equation}
\begin{aligned}
\left|\frac{\mu + |\xi|}{\mu \lambda} \right| |\xi| |d \lambda|
=&\left| \frac{ \theta a + ib + |\xi|}{ (\theta a + ib) \nu (\theta a +ib + |\xi|)  (\theta a +ib - |\xi|) } |\xi| 2 \nu( \theta a  + ib ) \right| |d b|\\
=& \frac{|\xi|}{|\theta a +i b - |\xi||}  |d b| \lesssim  |d b|.
\end{aligned}
   \llabel{EQ151}
     \end{equation}
Therefore, 
\begin{equation}
|R_\xi^{(2)}(t, y; z) | \leq \left|\int_{\Gamma} e^{\lambda t}e^{-\mu |y+z|} \frac{\mu + |\xi|}{\mu \lambda} |\xi| d \lambda\right| \leq C_0 e^{-\nu |\xi|^2t} e^{-\frac{|y+z|^2}{4 \nu t}} \int_{-\infty}^{\infty} e^{-\nu b^2 t } \, db.
   \llabel{EQ152}
     \end{equation}
When $\nu |\xi|^2 \geq 1$, a direct calculation gives
\begin{equation}
|R_\xi^{(1)}(t, y; z)|\leq 
2|\xi|e^{-|\xi||y+z|} \leq 2 \mu_0 e^{-\frac{1}{2}\mu_0|y+z|},
   \llabel{EQ153}
     \end{equation}
proving \eqref{EQ128} and~\eqref{EQ129} for~$k=0$.
The estimates for $k>0$ follows similarly to the previous case.
\end{proof}

\subsection{Duhamel's principle}
Finally, we address Duhamel's principle for the Stokes problem. We consider the equation 
\begin{subequations}
  \label{EQ154}
  \begin{align}
    \partial_{t}\omega_\xi - \nu\Delta_\xi\omega_\xi &= f\\
    -\nu(\partial_{z}+|\xi|)\omega_{\tau, \xi}|_{z=0} + \nu \xi|\xi|^{-1}\xi\cdot \omega_{\tau, \xi}|_{z=0}&= g  \\
    \omega_{3,\xi}|_{z=0} & = 0
    ,
  \end{align}
\end{subequations}
with $f,g$ sufficiently smooth,
and $\xi\in \mathbb{Z}^2$ is fixed.
From the discussion above, it is easy to check that Green's function for the equation has the form
  \begin{align}
  G_{\xi}(t,y;z)=
  \begin{pmatrix}
  G_{\xi,\tau}(t,y;z) &0\\
  0&G_{\xi,3}(t,y;z)
  \end{pmatrix},
   \llabel{EQ155}
  \end{align}
where
  \begin{align}
   G_{\xi,\tau}(t,y;z)=  \frac{1}{\sqrt{\nu t}}\left(e^{-\frac{|z-y|^2}{4\nu t}} + e^{-\frac{|z+y|^2}{4\nu t}} \right)e^{-\nu|\xi|^2t} I_2+  \int_{\Gamma} e^{\lambda t} \frac{\mu+ |\xi|}{ \mu  \lambda |\xi|} 
\begin{pmatrix}
\xi_2^2 & -\xi_1\xi_2\\
-\xi_1\xi_2 & \xi_1^2
\end{pmatrix}
 e^{-\mu|y+z|}\,d\lambda
   .
   \llabel{EQ156}
  \end{align}
Due to the Dirichlet boundary condition, there holds
  \begin{align}
   G_{\xi,3}(t,y;z)= \frac{1}{\sqrt{\nu t}}\left(e^{-\frac{|z-y|^2}{4\nu t}}- e^{-\frac{|z+y|^2}{4\nu t}} \right)e^{-\nu|\xi|^2t} . 
   \llabel{EQ157}
  \end{align}

Now we are ready to establish Duhamel's formula.

\cole
\begin{theorem}
\label{T4}
For any $T>0$, and for any $f \in L^{\infty}(0,T; L^2(\mathbb{R}_+))$ and $g \in L^{\infty}(0,T)$, the unique solution
in $L^{2}([0,T],H^{1}(\mathbb{R}_{+}))\cap L^{\infty}([0,T],L^{2}(\mathbb{R}_{+})) $
to the linear Stokes equation \eqref{EQ154}, with the initial data $\omega_{0,\xi}(z)$ in $L^1(\mathbb{R}_+)$, has the  representation
  \begin{align}
    \begin{split}
    \omega_{\xi}(t,y)= \int_0^{\infty}& G_{\xi}(t,y;z)\omega_{0,\xi}(z) \,dz + \int_0^t \int_0^{\infty} G_{\xi}(t-s,y;z) f(s,z) \,dz dt\\
    &+ \int_0^t G_{\xi}(s,y;z)|_{z=0} \begin{pmatrix} g(s) \\ 0\end{pmatrix}\,ds .
  \end{split}
   \label{EQ171}
  \end{align}
\end{theorem}
\colb

The proof of this theorem is obtained by standard methods, and thus
the proof is omitted.

As pointed out in the introduction, Green's function
representation
\eqref{EQ171}, together with \eqref{EQ199}--\eqref{EQ201} and
the upper bounds in Theorem~\ref{T3} have
as an immediate consequence the 3D analog of the result~\cite{KVW3},
which states that the inviscid limit holds, locally in time,
for the Navier-Stokes equations with data analytic close to the
boundary and Sobolev in the rest of the domain.
For the seminal work on the inviscid problem in domains with the
boundary, see \cite{K}; for various works on the Kato
criteria from \cite{K},
see \cite{CEIV,CKV,CLNV,CV,Ke,TW,Wa}.
For the results with analytic data analytic or analytic close to the boundary, see~\cite{BNNT,CS,FTZ,KLS,KVW3,KVW4,M2,SC1,SC2,WWZ}
for results on flows with symmetry, see \cite{LMN,LMPT,MT},
and for other results and a review~\cite{BW,GVMM,MM}.

\startnewsection{Green's function for general boundary condition}{sec05}
In this section, we compute Green's function for more general boundary conditions and then provide its upper bound.
To be specific, we consider the Stokes problem
 \begin{subequations}
  \label{EQ177}
  \begin{align}
    \partial_{t}\omega_\xi - \nu\Delta_\xi\omega_\xi &= N_\xi \llabel{EQ178}\\
    -\nu \partial_{z} \omega_{\tau,\xi}|_{z=0} + \nu D(\xi) \omega_{\tau, \xi}|_{z=0}&= -B_{\tau,\xi}
    \llabel{EQ179} \\
    \omega_{3,\xi}|_{z=0} & = 0,
    \llabel{EQ180}
    .
  \end{align}
 \end{subequations}

We assume that
$D(\xi)$ satisfies the following hypothesis.

\begin{Hypothesis}\label{H1}
{\rm
We assume that the matrix
$D(\xi)$ is
of the form
\begin{equation}
D(\xi)
=
\begin{pmatrix}
\a(\xi) & \g(\xi)\\
\g(\xi) & \b(\xi)
\end{pmatrix}
   \llabel{EQ181}
     \end{equation}
where the entries satisfy 
\begin{enumerate}
\item 
$\det{D}(\xi) =0$ \text{~and~}
\item
$ \a(\xi), b(\xi) \geq 0$ and $a(\xi) + \b(\xi)\le C_0|\xi|$, where $C_0>0$,
\end{enumerate}
for all $\xi\in\mathbb{Z}^2$.
}
\end{Hypothesis}

\subsection{Analytic semigroup}
We first consider the realization of the Laplace operator $\nu \Delta$ with a more general boundary condition.
\begin{definition}
Let $A_{\xi}$ be the realization of the Laplace operator $ \Delta_{\xi}$ with the boundary condition
  \begin{equation}
  - \partial_{z}u + D(\xi) u|_{z=0}= 0,
  \llabel{EQ182}
  \end{equation}
in $(L^2(\mathbb{R}_+))^2$. We define the domain of $A_{\xi}$ by
\begin{equation}
D(A_{\xi}):= \left\{u \in \left(H^2(\mathbb{R}_+)\right)^2 \text{and} - \partial_{z}u + D(\xi) u|_{z=0}= 0\right\}.
   \llabel{EQ183}
     \end{equation}
Similarly to Section~\ref{sec03}, we have the following resolvent estimates.
\end{definition}

\cole
\begin{theorem}
\label{T5}
Assuming that Hypothesis~\ref{H1} holds, the operator $A_{\xi}$ is sectorial. The resolvent set $\rho(A_{\xi})$ of $A_{\xi}$,
   contains $\mathbb{C} \setminus ((-\infty , -\nu |\xi|^2] \cup \{\nu((\alpha+\beta)^2-|\xi|^2)\})$. Moreover, we have
  \begin{align}\label{EQ184}
    \|R(\lambda; A_{\xi})\|_{L^2 \to L^2}  \lesssim_\theta \sqrt{2}{|\lambda+ \nu |\xi|^2|^{-1}} 
  \end{align}
and
  \begin{align}\label{EQ185}
    \|R(\lambda; A_{\xi})\|_{L^2 \to H^1}  \lesssim_\theta \sqrt{2}\nu^{-1/2} {|\lambda+ \nu |\xi|^2|^{-1}} 
    ,
  \end{align}
for
  \begin{equation}
		\lambda \in S_{\theta, \xi, \alpha, \beta}= \{\lambda' \in \mathbb{C}: |\lambda'| \ge C\nu \max((\alpha+\beta)^2, |\xi|^2),\ |\arg\lambda'|<\theta\}
		,
   \llabel{EQ245}		
  \end{equation}
for any $\theta\in(0, \pi)$ and $C$ sufficiently large.
\end{theorem}
\colb

\begin{proof}[Proof of Theorem~\ref{T5}]
We consider the resolvent equation 
\begin{align}
  \begin{split}
(\lambda - \nu \Delta_{\xi}) u =&f
\\
\gamma(\partial_3 u + D(\xi) u) =&0
  ,
  \end{split}
   \label{EQ186}
\end{align}
with $\lambda \in \mathbb{R}$ and $f\in L^2(\mathbb{R}_+)$.
As in Section~\ref{sec03}, let $\mu=\nu^{-1/2} \sqrt{\lambda + |\xi|^2 \nu}$.
As in Section~\ref{S3.1}, the function $v$ defined by
\begin{align}
v(y) = \int_{0}^{\infty} \frac{1}{2\nu \mu}  \left(e^{- \mu |z-y|} + e^{- \mu |z+y|}\right) {f}(z) dz,
   \llabel{EQ187}
\end{align}
solves $(\lambda - \nu \Delta_{\xi}) v =f$, with the homogeneous Neumann boundary condition. To match the boundary condition, we
write $u=v+w$ where
  \begin{align}
  w(y)=\int_0^{\infty}M(z) f(z) e^{-\mu y}dz
    .
   \label{EQ188}
  \end{align}
The function $M\colon \mathbb{R}_+\to \mathbb{C}^{2 \times 2}$ is chosen so that
\begin{align}
  \begin{split}
  (\lambda - \nu \Delta_{\xi}) w &=0
\\
    \gamma ( \partial_{3}(w+v) + D(\xi)(w+v))&=0.
  \end{split}
   \label{EQ38}
\end{align}
Following the arguments in Sections~\ref{sec03} and~\ref{sec04}, we obtain
  \begin{equation}
   -\gamma(\partial_3 w + D(\xi)w)
   =( \mu - D(\xi))\int_0^{\infty}M(z) f(z) dz
   \llabel{EQ190}
     \end{equation}
and 
  \begin{equation}
  \gamma(\partial_3 v + D(\xi)v)
  = D(\xi)\int_0^{\infty} \frac{e^{-\mu z}}{ \nu \mu}  f(z) dz
  ,
   \llabel{EQ191}
     \end{equation}
which then leads to
  \begin{align}
  \begin{split}
   &
   \int_0^{\infty}M(z) f(z) dz
   +
   ( \mu - D(\xi))^{-1}D(\xi)\int_0^{\infty} \frac{e^{-\mu z}}{ \nu \mu}  f(z) dz
   = 0
   .
  \end{split}
   \label{EQ39}
  \end{align}
It is easy to check that when $\lambda \neq \nu((\a+\b)^2 - |\xi|^2)$, the matrix $\mu- D(\xi)$ is invertible, and we have
\begin{equation}
(\mu-D(\xi))^{-1} D(\xi)= \frac{1}{\mu^2 - (\a+\b) \mu}\begin{pmatrix} \mu-\b& \g\\ \g & \mu-\a \end{pmatrix} \begin{pmatrix} \a & \g \\ \g & \b  \end{pmatrix}.
   \llabel{EQ192}
     \end{equation}
Using $\det D(\xi) = \alpha \beta-\gamma^2 =0$, we obtain 
\begin{equation}
(\mu-D(\xi))^{-1} D(\xi)=  \frac{1}{\mu - (\a+\b) } D(\xi)
   .
   \label{EQ193}
     \end{equation}
Using \eqref{EQ193} in \eqref{EQ39},
together with
  \begin{align}
  \begin{split}
     (\lambda - \nu \Delta_{\xi})
     \int_0^{\infty}M(z) f(z) e^{-\mu y}dz
     = 0
   \inin{\mathbb{R}_+}
  ,
  \end{split}
   \label{EQ95}
  \end{align}
resulting from \eqref{EQ188} and \eqref{EQ38}$_1$,
we obtain, after some computation,
  \begin{equation}
   M(z)= \frac{e^{-\mu z}}{\nu \mu(\mu-(\a+\beta))} D(\xi)
   .
   \llabel{EQ194}
  \end{equation}
Thus, for any $f \in (L^2(\mathbb{R}_+))^2$ and
$\lambda \in \mathbb{C} \setminus \{\nu((\a+\b)^2 - |\xi|^2)\}$, we obtain an explicit solution to the system~\eqref{EQ186}, which reads
  \begin{equation}
    u(y)
    = \int_{0}^{\infty} \frac{1}{2\nu \mu}  \left(e^{- \mu |z-y|} + e^{- \mu |z+y|}\right) {f}(z) dz + \int_0^{\infty} \frac{e^{-\mu |y+z|}}{\nu \mu(\mu-(\a+\beta))} D(\xi) f(z) dz.  
   \llabel{EQ195}
     \end{equation}
To derive \eqref{EQ184}, we observe that instead of~\eqref{EQ231}
we have
  \begin{equation}
    \|w\|_{L^2} \leq \frac{C_0}{|\lambda+ \nu |\xi|^2||\sqrt{1+\lambda/\nu |\xi|^2}-(\a+\b)/|\xi||}\|f\|_{L^2}.
   \llabel{EQ246}
\end{equation}
Hence, as in Theorem~\ref{T1}, upon choosing our contour $S_{\theta,\omega,\a,\b}$ carefully so that the term $\sqrt{1+\lambda/\nu |\xi|^2}-(\a+\b)/|\xi|$ is bounded from below, there exists a constant $C_{\theta,\omega,\a,\b}$ such that \eqref{EQ184} holds. The proof of \eqref{EQ185} is similar to~\eqref{EQ79.1} in Theorem~\ref{T1}.
\end{proof}

\subsection{Green's function}
From the previous considerations, it is easy to check that Green's function for the system~\eqref{EQ177} takes the form
\begin{align}
G_{\xi}(t,y;z)=
\begin{pmatrix}
  G_{\xi,\tau}(t,y;z) &0\\
  0&G_{\xi,3}(t,y;z)
\end{pmatrix}
  ,
   \label{EQ199}
\end{align}
where
\begin{align}
  \begin{split}
G_{\xi,\tau}(t,y;z)=&  \frac{1}{\sqrt{\nu t}}\left(e^{-\frac{|z-y|^2}{4\nu t}} + e^{-\frac{|z+y|^2}{4\nu t}} \right)e^{-\nu|\xi|^2t} 
+  \int_{\Gamma} e^{\lambda t} e^{-\mu|y+z|} \frac{D(\xi)}{\nu \mu(\mu- (\a+\b))}\,d\lambda
\\
=: & H_{\xi,\tau}(t,y;z) + R_{\xi,\tau}(t,y;z)
  \end{split}
   \llabel{EQ200}
\end{align}
and 
\begin{align}
G_{\xi,3}(t,y;z)= \frac{1}{\sqrt{\nu t}}\left(e^{-\frac{|z-y|^2}{4\nu t}}- e^{-\frac{|z+y|^2}{4\nu t}} \right)e^{-\nu|\xi|^2t} . 
   \label{EQ201}
\end{align}
Similarly to Theorem~\ref{T3}, we have the following statement.

\cole
\begin{theorem}
\label{T6}
The residual kernel $R_{\xi,\tau}$
may be decomposed as 
  $ R_{\xi,\tau}= R^{(1)}_{\xi,\tau}+ R^{(2)}_{\xi,\tau}$,
with the two kernels satisfying the bounds
  \begin{equation}
    |\partial_{z}^{k}R^{(1)}_{\xi,\tau}(t, y, z)|
    \lesssim \mu_0^{k+1} e^{-\theta_0\mu_0^{}(y+z)} e^{ \nu((\alpha+\beta)^2-|\xi|^2)t}
    \comma k\in{\mathbb N}_0
    \label{EQ202}
  \end{equation}
  and
  \begin{equation}
    |\partial_{z}^{k} R^{(2)}_{\xi,\tau}(t, y,z)|
    \lesssim \frac{1}{(\nu t)^{(k+1)/2 }}
    e^{-\frac{(y+z)^2}{\nu t}}e^{-\frac{\nu|\xi|^2t}{8}}
    \comma k\in{\mathbb N}_0
    ,
    \label{EQ203}
  \end{equation}
where 
  $\theta_0>0$ is a constant and
  the boundary remainder coefficient 
  is given by $\mu_0=\mu_0(\xi, \nu)= |\xi|+\sqrt{\nu}^{-1}$.
The implicit constants in \eqref{EQ202} and \eqref{EQ203} depend on~$k$.
\end{theorem}
\colb

\begin{remark}
Note that this result should be seen as a generalization of Theorem~\ref{T3}. In fact, if we further assume that $\alpha+\beta\le |\xi|$, then the 
two results 
coincide. 
\end{remark}

\begin{proof}[Proof of Theorem~\ref{T6}]
Since
\begin{equation}
\frac{1}{\nu \mu(\mu- (\a+\b))} = \frac{\mu + (\a+\b)}{\mu \left(\lambda + \nu (|\xi|^2- (\a+\b)^2)\right)}
   ,
   \llabel{EQ204}
     \end{equation}
we have
\begin{equation}
R_{\xi,\tau}(t,y;z):=  \int_{\Gamma} e^{\lambda t} e^{-\mu|y+z|}   \frac{\mu + \a+\b}{\mu \left(\lambda + \nu (|\xi|^2- (\a+\b)^2)\right)} D(\xi)\,d\lambda.
   \llabel{EQ205}
     \end{equation}
By Hypothesis~\ref{H1}, there holds 
\begin{equation}
|D(\xi)| \lesssim |\xi|.
   \llabel{EQ206}
     \end{equation}
We first consider the case when $\nu |\xi|^2 \leq 1 $ and decompose the contour as $\Gamma:= \Gamma_{-}\cup\Gamma_{+}\cup \Gammac$  with
  \begin{align}
    \begin{split}
    &
    \Gamma_{-}:= \left\{ \lambda= -\frac{1}{2}\nu |\xi|^2 + \nu (a^2-b^2) + 2 \nu i ab - i M,\ - b \in  [0,\infty) \right\},
    \\&
  \Gammac:= \left\{ \lambda =-\frac{1}{2}\nu|\xi|^2 + \nu a^2 + Me^{i\theta}, \ \theta\in [-\pi/2, \pi/2]  \right\},
    \\&
    \Gamma_{+}:= \left\{ \lambda= -\frac{1}{2}\nu |\xi|^2 + \nu (a^2-b^2) + 2 \nu i ab + i M,\ + b \in [0,\infty) \right\},
    \end{split}
   \llabel{EQ207}
     \end{align}
for some positive number $M$ large enough so that a ball $B$ centered at $\nu((\alpha+\beta)^2-|\xi|^2)$ with radius $2C_0$ is on the left-hand side of the contour $\Gamma$ , where $C_0>0$ is as in Hypothesis~\ref{H1} and
  \begin{equation}
   a=\frac{|y+z|}{2\nu t}
   .
   \label{EQ244}
  \end{equation}
Similarly to the proof of Theorem~\ref{T3}, we have
  \begin{equation}
  \re \mu \geq \theta_0 \mu_0 + \frac{a}{2}.
   \llabel{EQ209}
    \end{equation}
Recalling 
  \begin{align}
    \begin{split}
    R_{\xi}( \lambda ,y,z) =  e^{-\mu|y+z|}  \frac{\mu + \alpha +\beta}{ \mu(\lambda + \nu(|\xi|^2- (\a+\b)^2))} D(\xi),
  \end{split}
   \llabel{EQ210}
  \end{align}
we obtain
  \begin{equation}
  \left| R^{(1)}_\xi(t, y, z)\right|= \left|\int_{\Gammac} e^{\lambda t} R_{\xi}( \lambda ,y,z) \,d\lambda \right  |  
  \lesssim \int_{\Gammac}  e^{Mt}e^{\nu a^2 t}e^{-\frac{a}{2}|y+z|} e^{-\theta_0 \mu_0 |y+z|} \frac{|\mu|+ |\xi|}{ |\mu|  } |\xi| d \lambda
   \llabel{EQ211}
     \end{equation}
  where we used that $\nu|\xi|^2 \le 1$.
  Hence, we get 
  \begin{equation}
  \left|\int_{\Gammac} e^{\lambda t} R_{\xi}(t,y,z) \,d\lambda \right |  \lesssim \mu_0 e^{\theta_0 \mu_0 |y+z|}e^{\nu((\alpha+\beta)^2-|\xi|^2)t} .
   \llabel{EQ212}
     \end{equation}
  For the integral on $\Gamma_{-}\cup\Gamma_{+}$, we have
  \begin{equation}
  \left| e^{\lambda t} e^{-\mu|y+z|} \right| \leq e^{-\frac{1}{2}\nu |\xi|^2 t} e^{-\frac{|y+z|^2}{4\nu t}} e^{-\nu b^2 t},
   \llabel{EQ213}
     \end{equation}
which gives 
  \begin{align}
    \begin{split}
    \left| \int_{\Gamma_{-}\cup\Gamma_{+}} e^{\lambda t} R_{\xi}(\lambda,y,z) \,d\lambda \right | \leq  \int_{\Gamma_{-}\cup\Gamma_{+}}  e^{-\frac{1}{2}\nu |\xi|^2 t} e^{-\frac{|y+z|^2}{4\nu t}} e^{-\nu b^2 t} \left| \frac{\mu + \alpha +\beta}{ \mu(\lambda + \nu(|\xi|^2- (\a+\b)^2))}\right| |\xi|d \lambda.
  \end{split}
   \llabel{EQ214}
  \end{align}
Since $\lambda= -\frac{1}{2}\nu |\xi|^2 + \nu (a^2-b^2) + 2 \nu i ab \pm i M$ on $\Gamma_{-}\cup\Gamma_{+}$, we have
  \begin{equation}
  d \lambda = -2\nu i (a+ bi) db.
   \llabel{EQ215}
     \end{equation}
Similarly to the proof of Theorem~\ref{T3}, we have $|a+ib|  \lesssim |\mu|$. Therefore, we may write
  \begin{align}
    \begin{split}
    \left| \int_{\Gamma_{-}\cup\Gamma_{+}} e^{\lambda t} R_{\xi}(\lambda ,y,z) \,d\lambda \right| \lesssim&  \int_{-\infty}^\infty  e^{-\frac{1}{2}\nu |\xi|^2 t} e^{-\frac{|y+z|^2}{4\nu t}} e^{-\nu b^2 t} \frac{|\mu +\alpha + \beta| |\xi|}{|\mu^2-(\alpha + \beta )^2|} d b\\
    \lesssim&  \int_{-\infty}^\infty  e^{-\frac{1}{2}\nu |\xi|^2 t} e^{-\frac{|y+z|^2}{4\nu t}} e^{-\nu b^2 t} 
    \frac{|\xi|}{|\mu -\alpha-\beta|} d b
   .
  \end{split}
   \llabel{EQ216}
  \end{align}
Since $ 2C_0 \le |\lambda|$ on $\Gamma_{-}\cup\Gamma_{+}$, we have
  \begin{equation}
  \frac{|\xi|}{|\mu-|\alpha+\beta||} \le \frac{1}{|\sqrt{1+ \paren{\nu |\xi|^2}^{-1} \lambda}-C_0|} \lesssim  1.
   \llabel{EQ217}
     \end{equation}
Hence,
  \begin{align}
    \left| \int_{\Gamma_{-}\cup\Gamma_{+}} e^{\lambda t} R_{\xi}(\lambda ,y,z) \,d\lambda \right| \lesssim
    &  \int_{-\infty}^{\infty}  e^{-\frac{1}{2}\nu |\xi|^2 t} e^{-\frac{|y+z|^2}{4\nu t}} e^{-\nu b^2 t}\,   d b \lesssim (\nu t)^{-1/2}  e^{-\frac{1}{2}\nu |\xi|^2 t} e^{-\frac{|y+z|^2}{4\nu t}}.
   \llabel{EQ218}
  \end{align}
The proof for the derivatives is the same as in the proof of Theorem~\ref{T3}. 
  
We now consider the case when $\nu |\xi|^2 \geq 1$ using a contour 
  \begin{equation}
  \Gamma:= \left\{\lambda=  -\nu|\xi|^2 + \nu (\theta^2 a^2-b^2) + 2\nu i \theta a b,\ b \in \mathbb{R} \right\},
   \llabel{EQ219}
     \end{equation}
similar to the one in the proof of Theorem~\ref{T3},
where $\theta=1$ when $a/(\a+\b)\in (-\infty, 1/2)\cup (3/2, \infty)$ and $\theta=1/2$ when $a/(\a+\b) \in (1/2,3/2)$,
where $a$ is defined in \eqref{EQ244}.
Then
  \begin{align}
    \begin{split}
    R_{\xi}(t,y;z) :&= R_\xi^{(1)}(t, y; z) + R_\xi^{(2)}(t, y ;z)
    \\&= R_\xi^{(1)}(t, y; z) + \int_{\Gamma} e^{\lambda t}e^{-\mu |y+z|}
    \frac{\mu + \alpha +\beta}{ \mu(\lambda + \nu(|\xi|^2- (\a+\b)^2))} D(\xi)d \lambda  
  \end{split}
   \llabel{EQ220}
  \end{align}
  where $R_\xi^{(2)}(t, y; z)$ is the residue of $e^{\lambda t} R_{\xi}(\lambda ,y,z)$ at the pole $\lambda= \nu(-|\xi|^2+ (\a+\b)^2)$ if $\nu(-|\xi|^2+ (\a+\b)^2)$ is on the right of the contour $\Gamma$ and $0$ otherwise.
  On $\Gamma$, we have
  \begin{equation}
  \mu = \theta a + ib,\ \lambda = \nu (\theta a +ib + |\xi|)  (\theta a +ib - |\xi|) 
   \llabel{EQ221}
     \end{equation}
  and
  \begin{equation}
  \begin{aligned}
    \left|\frac{\mu + \a +\b}{\mu (\lambda+ \nu(|\xi|^2-(\a+\b)^2))} \right|  |d \lambda|
    =&\left| \frac{ \theta a + ib + \a+ \b}{ (\theta a + ib) \nu (\theta a +ib + \a +\b)  (\theta a +ib - \a-\b) }  2 \nu( \theta a  + ib ) \right| |d b|\\
    =& \frac{1}{|\theta a +i b - \a-\b|}  |d b| \lesssim  |d b|.
  \end{aligned}
   \llabel{EQ222}
     \end{equation}
Therefore, 
  \begin{align}
    \begin{split}
  |R_\xi^{(2)}(t, y; z) |
    &\leq \int_{\Gamma} e^{\lambda t}e^{-\mu |y+z|} \left| \frac{\mu + \a +\b}{\mu (\lambda+ \nu(|\xi|^2-(\a+\b)^2))} \right| |D(\xi)| |d \lambda|
  \\&
  \leq C_0 e^{-\nu |\xi|^2t} e^{-\frac{|y+z|^2}{4 \nu t}} \int_{\mathbb{R}} e^{-\nu b^2 t } db.
  \end{split}
   \llabel{EQ223}
     \end{align}
  When $\nu |\xi|^2 \geq 1$, a direct calculation gives
  \begin{equation}
  |R_\xi^{(1)}(t, y; z)|\leq e^{ \nu((\alpha+\beta)^2-|\xi|^2)t}
  2|\xi|e^{-|\xi||y+z|} \leq 2 \mu_0 e^{-\frac{1}{2}\mu_0|y+z|}e^{ \nu((\alpha+\beta)^2-|\xi|^2)t},
   \llabel{EQ224}
     \end{equation}
proving \eqref{EQ202} and \eqref{EQ203} for~$k=0$.
The estimates for $k>0$ follows similarly as in the previous case,
concluding the proof.
\end{proof}


%
%

\section{Acknowledgment}
IK was supported in part by the NSF grant DMS-2205493,
while FW was supported in part by the National Natural Science Foundation of China (No.~12101396, 12331008, and 12161141004).

\end{document}